\documentclass[12pt,oneside,reqno]{amsart}
\usepackage{txfonts}
\usepackage{bbm}
\usepackage{graphicx}
\usepackage{stmaryrd}
\usepackage{enumerate,amsmath,amssymb,amsthm}
\numberwithin{equation}{section}
\usepackage{amsthm}
\usepackage{amsfonts,mathrsfs}
\usepackage{enumerate}
\usepackage{color}
\usepackage{xcolor}
\usepackage{verbatim}

\usepackage[colorlinks=true]{hyperref}
\hypersetup{
	linkcolor=blue,          
	citecolor=red,        
	filecolor=blue,      
	urlcolor=cyan
}

\newtheorem{theorem}{Theorem}[section]
\newtheorem{lemma}[theorem]{Lemma}%
\newtheorem{remark}[theorem]{Remark}%

\newtheorem{definition}[theorem]{Definition}%

\def\mbb{\mathbb}

\def\mcl{\mathcal}

\def\mbf{\mathbf}
\def\e{{\mathrm{e}}}
\def\d{\mathrm{d}}

\def\S{\mathbb{S}}
\def\P{\mathbb{P}}

\allowdisplaybreaks[4]
\begin{document}
	\title{Well-posedness and Feller property for functional stochastic Hamiltonian systems with singular coefficients and state-dependent switching$^{\star}$}
	
	\date{}
	\author{Fubao Xi$^{1}$, Yafei Zhai$^{2}$$^{*}$ and Zuozheng Zhang$^{2}$}
\dedicatory{
	$^1$
	Institute for Mathematics and Interdisciplinary Sciences, Beijing Institute of Technology, Zhuhai, 519088, P.R.China\\
	$^2$ School of Mathematics and Statistics,
	Beijing Institute of Technology, 
	Beijing, 100081,  P.R.China}
	\thanks{$^{\star}$  Supported in part by the National Natural Science Foundation of China under Grant No. 12071031.}
	\thanks{$^{*}$ Corresponding author.}
	\thanks{E-mail: xifb@bit.edu.cn (F. Xi), yafeizhai@bit.edu.cn(Y. Zhai), zuozhengzhang@mail.bnu.edu.cn(Z. Zhang) }	
	
	
	\begin{abstract}
		This work focuses on a class of functional stochastic Hamiltonian systems with singular coefficients and state-dependent switching, in which the switching process has a countably infinite state space. First, by Girsanov’s transformation, we establish the martingale solution of the system in each fixed switching case. Then, since the two components of the system are intertwined and correlated, we investigate the special case when the discrete component is independent of the continuous component. Based on these results, we obtain the well-posedness of the system with the aid of a martingale process. Finally, in order to establish the Feller property, we take advantage of appropriate Radon-Nikodym derivative and introduce a vector valued elliptic equation to handle the effect of the singular coefficients.
		
	\end{abstract}
	\maketitle
	\rm
	
	\section{Introduction}
	Motivated by the increasing need of modeling complex systems, designing optimal control, and conducting optimization tasks, the hybrid system has been proposed and developed in many works \cite{Mao2006,Yin2010}. It allows to describe the random environment that is otherwise not representable by the traditional differential equations. Additionally, the real system often depends not only on the present state but also on some of its past history. Functional stochastic differential equations have received considerable attention due to their broad applicability in diverse fields, including process control, biomedical sciences, epidemics, transport, communication networks, and population dynamics \cite{Kushner2008,Gopalsamy1992,Hale1993}. Moreover, when systems are influenced by the whole past history, it is necessary to consider systems with infinite delay. Furthermore, Hamiltonian systems subject to random perturbations are particularly interesting as many real mechanical and physical systems are unavoidably influenced by random noises. The noise is degenerate because it acts directly only on the velocity but not position. Hamiltonian systems have a wide range of applications and are commonly used as models for virtually all fields of mechanics and physics \cite{Feng2013,zhangX2010}. Based on the above considerations, this paper focuses on a functional stochastic Hamiltonian system with state-dependent switching. Suppose $Z(t):=(X(t), Y(t))$ is governed by the following stochastic differential equation
	\begin{equation}\label{eq:zaikan1}
		\left\{
		\begin{aligned}
			\text{d}X(t)&=\left(aX(t)+bY(t)\right)\text{d}t,\\
			\text{d}Y(t)&=\left[ b_1(X_t,Y_t,\Theta(t))+b_2(X(t),Y(t),\Theta(t))\right]\text{d}t+\sigma(X(t),Y(t),\Theta(t))\text{d}B(t),
		\end{aligned}
		\right.
	\end{equation}
	with initial data $Z_0=\varphi \in \mcl C_r^{2d}$ and $\Theta(0)=i \in \mbb S:=\{1,2,\cdots\}$, where $a\geq0$ and $b\neq0$, $b_1:\mcl C_r^{2d} \times \mbb S\rightarrow \mathbb{R}^d,b_2: \mbb R^{2d}\times \mbb S \rightarrow \mathbb{R}^d$ and  $\sigma: \mbb R^{2d} \times \mbb S \rightarrow \mathbb{R}^d \otimes \mathbb{R}^d$ are Borel measurable functions. $X_t(\theta):=X(t+\theta)$ and  $Y_t(\theta):=Y(t+\theta),-\infty<\theta \leq 0$ are the segment process of $X(t)$ and $Y(t)$ respectively, $B(t)$ is an $d$-dimensional standard Brownian motion on a complete filtration probability space $\left(\Omega, \mathscr{F},\left(\mathscr{F}_t\right)_{t \geq 0}, \mathbb{P}\right)$, and for a fixed number $r\in (0,\infty)$,
	$$
	\begin{gathered}
	{ 	\mcl C_r^{2d}:=\left\{\varphi \in C((-\infty, 0] ; \mathbb{R}^{2d}) ; \lim _{\theta \rightarrow-\infty} \mathrm{e}^{r \theta} \varphi(\theta) \text { exists in } \mathbb{R}^{2d}\right\}.}
	\end{gathered}
	$$
	Then $(\mathcal{C}_r^{2d},\|\cdot\|_r)$ is a Polish space (see \cite{hino1991} for more details). The norm $\|\cdot\|_r$ fits the intuition of exponential decay with regard to the influence of history; that is, the contribution to the norm from the history at time $\theta<0$ has a minus exponential discount $\mathrm{e}^{r \theta}$.
	
	The second component $\Theta(t)$ is a right continuous random process with a countably infinite state space $\mbb S$ given by
	\begin{equation}\label{eq:zaikan2}
		\mbb {P}\{\Theta(t+\Delta)=l \mid \Theta(t)=k, Z_t=\varphi \}= \begin{cases}q_{kl}(\varphi) \Delta+o(\Delta), & l \neq k ,\\ 1+q_{kk}(\varphi) \Delta+o(\Delta), & l=k,\end{cases}
	\end{equation}
	provided $\Delta \downarrow 0$, where $q_{kl}(\varphi) \geq 0$ is the transition rate from $k$ to $l$, if $l \neq k$, and  $q_{kk}(\varphi)=-\sum_{l \neq k} q_{kl}(\varphi)$ for every $k\in\mbb S$.
	
	As a typical model of degenerate diffusion system, the stochastic Hamiltonian system has been studied by many authors. For example, Wu \cite{wu2001} considers the large deviation principle for stochastic Hamiltonian systems by constructing an appropriate Lyapunov test function. Xi, Zhu and Wu \cite{xi2021} investigate the strong Feller property and exponential ergodicity for stochastic Hamiltonian systems. Wang \cite{wang2017} studies the hypercontractivity for the Markov semigroup associated with a class of stochastic Hamiltonian systems. However, the stochastic process of these articles does not involve delay argument. This means that the future state of the system is determined only by the present and is not affected by the past history. But a more realistic model usually includes some of the past history of the system such as predator-prey models. We call this model functional stochastic differential equations (FSDEs). FSDEs provide powerful mathematical tools in modeling and analyzing complex memory-dependent dynamical systems. The theory for such processes has been extensively developed recently (see \cite{Mohammed1984,Mao2007,McShane1969,Mao1996,mao2000,Shen2006}). For literature on existence of solution for FSDEs, we refer to Mohammed \cite{Mohammed1984} for FSDEs with finite delay and Wu et. al. \cite{wu2017} for FSDEs with infinite delay on appropriate phase space. Moreover, they also show that the segment process admits  Markov property and  ergodicity. In addition, Bao et al. \cite{bao14,BAO2014} investigate ergodicity under different conditions. Zhai and Xi \cite{Zhai2024} explore exponential ergodicity for several kinds of FSDEs. These papers study the non-degenerate FSDEs with regular coefficients. So far, the research on the degenerate FSDEs with singular coefficients is limited. 
	
	On the other hand, in many real-world systems, the classical models using solely stochastic differential equations are frequently inadequate since the behavior of the complex system would be affected by the random environment and other stochastic influences. As a result, the switching model is more versatile and has a wider range of applicability. Recently, regime-switching diffusion processes are extensively proposed and investigated in many fields (see \cite{Mao2006,Yin2010,Bao2016,Bao2020,mao2013}). Precisely, the regime-switching diffusion process has two components, including continuous component and discrete component. When discrete component is independent of continuous component, the process is called a diffusion process with Markovian switching. When discrete component is dependent of continuous component, the process is called a diffusion process with state-dependent switching. Nevertheless, owing to the addition of the state-dependent switching, the problem is more complicated. So far, there are several approaches to explore the existence and uniqueness of strong solutions for regime-switching diffusion processes; see, for instance,  Xi \cite{xi2009} by successive construction methods and Xi and Zhu \cite{Xi2017} via ``interlacing procedure". In addition, Nguyen and Yin \cite{Nguyen2016} also establish the existence and uniqueness of strong solutions of the associated stochastic differential equation under different conditions. However, there are few studies on the weak solutions of such systems.
	
	In this paper, we investigate a class of functional stochastic Hamiltonian systems with singular coefficients and state-dependent switching \eqref{eq:zaikan1}-\eqref{eq:zaikan2}. To be specific, we are interested in the situation that the coefficient $b_2(\cdot, k)$ is H\"older continuous in \eqref{eq:zaikan1} for every $k\in\mbb S$.  To obtain the existence and uniqueness of the martingale solution for the process \eqref{eq:zaikan1}-\eqref{eq:zaikan2} (see Definition \ref{defn1} below), we first show the existence and uniqueness of the martingale solution for the process in each fixed switching by Girsanov’s transformation. Then, since the two components of functional stochastic Hamiltonian systems \eqref{eq:zaikan1}-\eqref{eq:zaikan2} are intertwined and correlated, we look at the special case when the discrete component $\Theta(t)$ is independent of the continuous component $Z_t$, and obtain the existence and uniqueness of the martingale solution to \eqref{eq:zaikan1}-\eqref{eq:zaikan2} with the aid of a martingale process $M_t$ defined in \eqref{eq:buyi} below. Finally, we establish the Feller property of the process $(Z_t, \Theta(t))$. For this purpose, we introduce a vector valued elliptic equation \eqref{eq:elliptic} to handle the effect of the singular coefficients $b_2$. For more details concerning the transformation, the reader can refer to \cite{Flandoli2010}.
	
	The rest of the paper is arranged as follows. In Section 2, we present some frequently used notations as well as necessary assumptions. In Section 3, we show the existence and uniqueness of the martingale solution for \eqref{eq:zaikan} below. The well-posedness of the system \eqref{eq:zaikan1}-\eqref{eq:zaikan2} is divided into two parts. In Section 4, we investigate the special case, i.e., the Markovian switching case. In Section 5, we deal with the general case, i.e., the state-dependent switching case. In Section 6, we establish the Feller property. Finally, we make the concluding remarks in Section 7.
	
	\section{Preliminaries}
	To facilitate the later presentation, we introduce some frequently used notations here.  For each strictly positive integer ${n}$, let ${\mathbb{R}^{n}}$ be an ${n}$-dimensional Euclidean space endowed with the inner product ${\langle u, v\rangle:=\sum_{i=1}^{n} u^{i} v^{i}}$ for ${u,v \in \mathbb{R}^{n}}$ and the Euclidean norm ${|u|:=\langle u, u\rangle^{1 / 2}}$ for ${u \in \mathbb{R}^{n}}$. Let ${\mathbb{R}^{n} \otimes \mathbb{R}^{m}}$ denote the collection of all ${n \times m}$ matrices with real entries. Given an ${n \times m}$ matrix ${A}$, $\|A\|_{\rm HS}$ denotes the Hilbert-Schmidt  norm of ${A}$. Let $A$ be a vector or a matrix and $A^{T}$ denotes its transpose.
	
	Define a metric \( \lambda(\cdot, \cdot) \) on \( \mathcal{C}_r^{2 d} \times \mathbb{S} \) as
	\[
	\lambda((\varphi, k),(\psi,l))=\|\varphi-\psi\|_r+\textbf{1}_{\{k \neq l\}} ,
	\]
	for $(\varphi,k),(\psi,l)\in\mathcal{C}_r^{2 d} \times \mathbb{S}$.  Let \( \mathcal{B}(\mcl C_r^{2d} \times \mathbb{S}) \) be the Borel \( \sigma \)-algebra on \( \mcl C_r^{2d} \times \mathbb{S} \). Hence, we know that \( (\mcl C_r^{2d} \times \mathbb{S}, \lambda(\cdot, \cdot), \mathcal{B}(\mcl C_r^{2d} \times \mathbb{S})) \) is a locally compact and separable metric space. As usual, let \( C([0, \infty), \mcl C_r^{2d}) \) be the continuous function space endowed with the sup norm topology and \( D([0, \infty), \mathbb{S}) \) be the càdlàg space endowed with the Skorohod topology. Moreover, let \( \Omega:=C([0, \infty), \mcl C_r^{2d}) \times D([0, \infty), \mathbb{S}) \) be endowed with the product topology of the sup norm topology on \( C([0, \infty), \mcl C_r^{2d}) \) and the Skorohod topology on \( D([0, \infty), \mbb S) \). Let \( \mathcal{F}_{t} \) be the \( \sigma \)-field generated by the cylindrical sets on \( \Omega \) up to time \( t \) and set \( \mathcal{F}=\bigvee_{t=0}^{\infty} \mathcal{F}_{t} \). Next, let \( C_{c}^{2}(\mbb{R}^{2 d} \times \mathbb{S}) \) denote the family of functions defined on \( \mbb{R}^{2 d} \times \mathbb{S} \) such that \( f(\cdot, k) \in C_{c}^{2}(\mbb{R}^{2 d}) \) for each \( k \in \mathbb{S} \), and \( f(z, \cdot) \) is a bounded function on \( \mbb S \) for each \( z \in \mbb{R}^{2 d} \), where \( C_{c}^{2}(\mbb{R}^{2 d}) \) denotes the family of functions defined on \( \mbb{R}^{2 d} \) which are twice continuously  differentiable and have compact support. When there is no special description in what follows, we denote $\varphi(0)=(\varphi_1(0),\varphi_2(0))$, $\psi(0)=(\psi_1(0),\psi_2(0))$ and  $\varphi=(\varphi_1,\varphi_2)$, $\psi=(\psi_1,\psi_2)$ for $\varphi,\psi\in \mcl C_r^{2d}$. Additionally, $C>0$ is a generic constant whose value may be different from line to line in the sequel.
	
	We recall the notion of martingale problem for the operator \( \mathcal{A} \) corresponding to the system \eqref{eq:zaikan1}-\eqref{eq:zaikan2} in this section.	Define the operator $\mcl A$ corresponding to the system \eqref{eq:zaikan1}-\eqref{eq:zaikan2} as follows:
	\begin{equation}\label{eq:operator}
		(\mathcal{A} f)(\varphi, k):=\mathbb{L}_{k}(\varphi) f(\varphi(0),k)+Q(\varphi) f(\varphi(0), k) ,
	\end{equation} 
	for any $f\in C_{c}^{2}(\mbb{R}^{2 d} \times \mathbb{S})$ and $(\varphi,k)\in\mathcal{C}_r^{2 d} \times \mathbb{S}$. Here, for each \( k \in \mathbb{S}\), \(\mathbb{L}_{k}(\varphi) \) is a differential operator defined by
	\begin{equation}\label{eq:operatork}
		\begin{aligned}
			\mathbb{L}_{k}(\varphi)f(\varphi(0), k) :=&\frac{1}{2} \operatorname{tr}\left(a(\varphi(0), k) \nabla_{y}^{2} f(\varphi(0), k)\right)+\left\langle a\varphi_1(0)+b\varphi_2(0), \nabla_{x} f(\varphi(0), k)\right\rangle \\
			&+\left\langle b_1(\varphi, k)+b_2(\varphi(0), k), \nabla_{y} f(\varphi(0), k)\right\rangle,
		\end{aligned}
	\end{equation}
	and the switching operator \( Q(\varphi) \) is defined as follows:
	\begin{equation}\label{eq:operatorq}
		Q(\varphi) f(\varphi(0), k):=\sum_{l \in \S} q_{kl}(\varphi)(f(\varphi(0), l)-f(\varphi(0),  k)) .
	\end{equation}
	Here and hereafter, \( a(\varphi(0), k)=\sigma(\varphi(0), k) \sigma(\varphi(0), k)^{T}, \nabla_x \) and $\nabla_y$ represent the gradient of functions with respect to the variables $x$ and $y$ respectively,  and  \( \nabla_y^{2} \) denotes the Hessian matrix of functions with respect to the variable $y$. 
	\begin{definition}\label{defn1}
		For a given \( (\varphi, k) \in \mathcal{C}_r^{2 d} \times \mathbb{S} \), we say a probability measure \( \mathbb{P}^{(\varphi, k)} \) on \(  C([0, \infty), \mcl C_r^{2d})\times D([0, \infty), \mathbb{S}) \) is a solution to the martingale problem for the operator \( \mathcal{A} \) starting from \( (\varphi, k) \), if \( \mathbb{P}^{(\varphi, k)}\{(Z_0, \Theta(0))=(\varphi, k)\}=1 \) and for each function \( f \in C_{c}^{2}(\mbb{R}^{2 d} \times \S) \),
		\begin{equation}\label{eq:Mf}
			M_{t}^{(f)}(Z_{\cdot}, \Theta(\cdot)):=f(Z(t), \Theta(t))-f(Z(0), \Theta(0))-\int_{0}^{t} (\mathcal{A} f)(Z_s, \Theta(s)) \mathrm{d} s
		\end{equation}
		is an \( \left\{\mathcal{F}_{t}\right\} \)-martingale with respect to \( \mathbb{P}^{(\varphi, k)} \), where \( (Z_t, \Theta(t)) \) is the coordinate process defined by \( (Z_t( \omega), \Theta(t, \omega))=\omega(t) \in \mcl C_r^{2 d} \times \mathbb{S} \) for all \( t \geq 0 \) and \( \omega \in \Omega \). Sometimes, we simply say that the probability measure \( \mathbb{P}^{(\varphi, k)} \) is a martingale solution for the operator \( \mathcal{A} \) starting from \( (\varphi, k) \)  to the system \eqref{eq:zaikan1}-\eqref{eq:zaikan2} with initial data \( (\varphi, k) \).	\end{definition}
	Before stating the main results, we firstly impose some preliminary assumptions in this section. These assumptions will be crucial for our later developments.
	
	\begin{enumerate} 
		\item [\textbf{(A1)}] For any $k\in\S$,  $b_1(\cdot,k)$ is bounded. 
		\item [\textbf{(A2)}] $b_2$ is H\"older continuous with the H\"older exponent $\alpha\in(0,1)$ and the H\"older constant $L_2>0$, i.e., 
		\begin{equation}
			|b_2(\varphi(0),k)-b_2(\psi(0),k)|\leq L_2|\varphi(0)-\psi(0)|^{\alpha}, \quad \varphi,\psi\in\mcl C_r^{2d}, k\in\S.
		\end{equation}
		\item [\textbf{(A3)}] For any $k\in\S$, the random perturbation $\sigma(\cdot,k)$ is symmetric. 
		Moreover, there exists a positive constant $\hat{\sigma}_k$ such that $0<\sigma(x,k)\leq\hat{\sigma}_k I$ over $\mbb R^{2d}$, where $\sigma(x,k)>0$ means that it is strictly positive definite and $I$ is the $d$-dimensional identity matrix.
		\item [\textbf{(A4)}]The formal operator of the switching process $Q(\varphi):=(q_{kl}(\varphi))$ is a matrix-valued measurable function on $\mcl C_r^{2d}$ and there exists a constant $H\geq 0$ such that
		\begin{equation}\label{eq:asnq}
			\sup_{k\in\mbb S}\sum_{l\neq k}\sup_{\varphi\in\mcl C_r^{2d}}q_{kl}(\varphi)\leq H.
		\end{equation}
	\end{enumerate}

	\section{Martingale solution: fixed switching case}
	In order to prove the existence and uniqueness of the martingale solution for the system \eqref{eq:zaikan1}-\eqref{eq:zaikan2}, we now introduce a family of processes. For each $k \in \S$, let $Z^{(k) }(t) := (X^{(k)}(t), Y^{(k)}(t))$ satisfy the following functional stochastic differential equation
	\begin{equation}\label{eq:zaikan}
		\left\{
		\begin{aligned}
			\text{d}X^{(k)}(t)&=\left(aX^{(k)}(t)+bY^{(k)}(t)\right)\text{d}t,\\
			\text{d}Y^{(k)}(t)&=\left[ b_1(X^{(k)}_t,Y^{(k)}_t,k)+b_2(X^{(k)}(t),Y^{(k)}(t),k)\right]\text{d}t+\sigma(X^{(k)}(t),Y^{(k)}(t),k)\text{d}B(t).
		\end{aligned}
		\right.
	\end{equation}

	In this section, we shall establish the existence and uniqueness of the martingale solution for \eqref{eq:zaikan} based on Girsanov's transformation. To achieve this goal, we introduce the following reference functional stochastic differential equation:
	\begin{equation}\label{eq:fuzhu}
		\left\{
		\begin{aligned}
			\text{d}\widetilde{X}^{(k)}(t)&=\left(a\widetilde{X}^{(k)}(t)+b\widetilde{Y}^{(k)}(t)\right)\text{d}t,\\
			\text{d}\widetilde{Y}^{(k)}(t)&=b_1(\widetilde{X}^{(k)}_t,\widetilde{Y}^{(k)}_t,k)\text{d}t+\sigma(\widetilde{X}^{(k)}(t),\widetilde{Y}^{(k)}(t),k)\text{d}B(t).
		\end{aligned}
		\right.
	\end{equation}
	The differential operator \({\mathcal{\widetilde L}}^{(k)}(\varphi) \) of this equation is defined as follows: for each \( f \in C_{c}^{2}(\mbb{R}^{2 d}) \),
	\begin{equation}
		\begin{aligned}
			\mathcal{\widetilde L}^{(k)}(\varphi) f(\varphi(0)):=&\frac{1}{2} \operatorname{tr}\left(a(\varphi(0), k) \nabla_{y}^{2} f(\varphi(0))\right)+\left\langle a\varphi_1(0)+b\varphi_2(0), \nabla_{x} f(\varphi(0))\right\rangle \\
			&+\left\langle b_1(\varphi, k), \nabla_{y} f(\varphi(0))\right\rangle.
		\end{aligned}
	\end{equation}
	
	{Under $\textbf{(A1)}$ and $\textbf{(A3)}$, using a similar argument with \cite[Lemma 1.1 ]{wu2001}, it is easy to check that  \eqref{eq:fuzhu} has a unique weak solution $\widetilde{Z}^{(k)}(t):=(\widetilde{X}^{(k)}(t), \widetilde{Y}^{(k)}(t))$ with the initial value $\widetilde{Z}_0^{(k)}=\varphi$.} Let  $\widetilde{\mbb P}^{(k)}$ be the corresponding law of the solution. Before we move forward to achieve the primary goal, let us prepare a warm-up lemma.
	
	\begin{lemma}\label{lem:Novikov}
		Suppose that $\textbf{(A1)-(A3)}$ hold. Then, for any $\lambda, T>0$,
		\begin{align*}
			\mbb E_{\widetilde{\mbb P}^{(k)}}\exp\left\{\lambda\int_{0}^{T}\left|\sigma^{-1}(\widetilde{Z}^{(k)}(t),k)b_2(\widetilde{Z}^{(k)}(t),k)\right|^2\d t\right\}<\infty.
		\end{align*}
	\end{lemma}
	\begin{proof}
		Simply write  $\mbb E=\mbb E_{\widetilde{\mbb P}^{(k)}}$. First, we show 
		\begin{equation}\label{eq:bounded}
			\mbb E\exp\left\{\mu\int_{0}^{T}\left(|\widetilde{X}^{(k)}(t)|^2+|\widetilde{Y}^{(k)}(t)|^2\right)\d t\right\}<\infty,
		\end{equation}
		where $0<\mu<\frac{\e^{-(2a+|b|+1)T-1}}{2d\hat{\sigma}_k^2T^2}$. Applying Jensen's inequality, we have
		\begin{equation}\label{eq:boundedguji}
			\begin{aligned}
				&\mbb E\exp\left\{\mu\int_{0}^{T}\left(|\widetilde{X}^{(k)}(t)|^2+|\widetilde{Y}^{(k)}(t)|^2\right)\d t\right\}\leq \frac{1}{T}\int_{0}^{T}\mbb E\e^{\mu T(|\widetilde{X}^{(k)}(t)|^2+|\widetilde{Y}^{(k)}(t)|^2)}\d t.
			\end{aligned}
		\end{equation}
		Next, by It\^o's formula, it follows from $\textbf{(A1)}$ and $\textbf{(A3)}$ that for $\gamma>0$
		\begin{align}
			\d\nonumber& \e^{-\gamma t}(|\widetilde{X}^{(k)}(t)|^2+|\widetilde{Y}^{(k)}(t)|^2)\\\nonumber
			&=\e^{-\gamma t}\Big\{-\gamma(|\widetilde{X}^{(k)}(t)|^2+|\widetilde{Y}^{(k)}(t)|^2)+2\langle \widetilde{X}^{(k)}(t), a\widetilde{X}^{(k)}(t)+b\widetilde{Y}^{(k)}(t)\rangle+2\langle \widetilde{Y}^{(k)}(t), b_1(\widetilde{X}^{(k)}_t, \widetilde{Y}^{(k)}_t,k)\rangle\\\nonumber
			&\quad+\|\sigma(\widetilde{X}^{(k)}(t), \widetilde{Y}^{(k)}(t), k)\|_{\rm HS}^2\Big\}\d t+2\e^{-\gamma t}\langle \widetilde{Y}^{(k)}(t), \sigma(\widetilde{X}^{(k)}(t), \widetilde{Y}^{(k)}(t),k)\d B(t)\rangle\\\nonumber
			&\leq \e^{-\gamma t}\Big\{-\gamma(|\widetilde{X}^{(k)}(t)|^2+|\widetilde{Y}^{(k)}(t)|^2)+2a|\widetilde{X}^{(k)}(t)|^2+|b|(|\widetilde{X}^{(k)}(t)|^2+|\widetilde{Y}^{(k)}(t)|^2)+{C_{k, \hat{\sigma}_k}}+|\widetilde{Y}^{(k)}(t)|^2\Big\}\d t\\\nonumber
			&\quad+2\e^{-\gamma t}\langle \widetilde{Y}^{(k)}(t), \sigma(\widetilde{X}^{(k)}(t), \widetilde{Y}^{(k)}(t),k)\d B(t)\rangle\\\nonumber
			&\leq \e^{-\gamma t}\left\{C_{k, \hat{\sigma}_k}-(\gamma-2a-|b|-1)(|\widetilde{X}^{(k)}(t)|^2+|\widetilde{Y}^{(k)}(t)|^2)\right\}\d t+2\e^{-\gamma t}\langle \widetilde{Y}^{(k)}(t), \sigma(\widetilde{X}^{(k)}(t), \widetilde{Y}^{(k)}(t),k)\d B(t)\rangle,\nonumber
		\end{align}
		where $C_{k, \hat{\sigma}_k}$ is a positive constant depending on $k$ and $\hat{\sigma}_k$. Set $\Gamma(t):=\varepsilon\e^{-\gamma t}(|\widetilde{X}^{(k)}(t)|^2+|\widetilde{Y}^{(k)}(t)|^2)$. So, by the above equation, It\^o's formula and $\textbf{(A3)}$, we deduce that
		\begin{equation*}\label{eq:chongzhishu}
			\begin{aligned}
				\d \e^{\Gamma(t)}&=\e^{\Gamma(t)}\left(\varepsilon\d(\e^{-\gamma t}(|\widetilde{X}^{(k)}(t)|^2+|\widetilde{Y}^{(k)}(t)|^2))\right.\\
				&\quad\left.+\frac{\varepsilon^2}{2}
				\d\langle\e^{-\gamma t}(|\widetilde{X}^{(k)}(t)|^2+|\widetilde{Y}^{(k)}(t)|^2),\e^{-\gamma t}(|\widetilde{X}^{(k)}(t)|^2+|\widetilde{Y}^{(k)}(t)|^2)\rangle\right)\\
				&\leq\e^{\Gamma(t)}\left\{\varepsilon \e^{-\gamma t}\left(C_{k, \hat{\sigma}_k}-(\gamma-2a-|b|-1)(|\widetilde{X}^{(k)}(t)|^2+|\widetilde{Y}^{(k)}(t)|^2)\right)\d t \right.\\
				&\left.\quad+2\varepsilon^2\e^{-2\gamma t}|\sigma(\widetilde{X}^{(k)}(t),\widetilde{Y}^{(k)}(t),k)^T\widetilde{Y}^{(k)}(t)|^2\d t+2\varepsilon\e^{-\gamma t}\langle \widetilde{Y}^{(k)}(t), \sigma(\widetilde{X}^{(k)}(t), \widetilde{Y}^{(k)}(t),k)\d B(t)\rangle\right\}\\
				&\leq -\varepsilon(\gamma-2a-|b|-1-2\varepsilon d\hat{\sigma}^2)\e^{-\gamma t}\e^{\Gamma(t)}\left(|\widetilde{X}^{(k)}(t)|^2+|\widetilde{Y}^{(k)}(t)|^2\right)\d t +C_{k, \hat{\sigma}_k}\varepsilon\e^{\Gamma(t)}\d t\\
				& \quad+2\varepsilon\e^{-\gamma t}\e^{\Gamma(t)}\langle \widetilde{Y}^{(k)}(t), \sigma(\widetilde{X}^{(k)}(t), \widetilde{Y}^{(k)}(t),k)\d B(t)\rangle, \quad \gamma>0, \ \varepsilon>0.
			\end{aligned}
		\end{equation*} 
		Thus, for any $\gamma>2a+|b|+1+2\varepsilon d\hat{\sigma}_k^2$, we have
		$
		\mbb E\e^{\Gamma(t)}\leq \e^{\varepsilon|\varphi(0)|^2}+C_{k, \hat{\sigma}_k}\varepsilon\int_{0}^{t}\mbb E \e^{\Gamma(s)}\d s.
		$
		Furthermore, by Gronwall's inequality, we deduce   
		$
		\mbb E\e^{\Gamma(t)}\leq \e^{\varepsilon|\varphi(0)|^2+C_{k, \hat{\sigma}_k}\varepsilon t},
		$
		which yields that 
		\begin{equation}\label{eq:Esup}
			\sup_{0\leq t \leq T}\mbb E\left(\exp\{\varepsilon\e^{-\gamma T}(|\widetilde{X}^{(k)}(t)|^2+|\widetilde{Y}^{(k)}(t)|^2)\}\right)\leq \e^{\varepsilon|\varphi(0)|^2+C_{k, \hat{\sigma}_k}\varepsilon T},
		\end{equation}
		where $\gamma>2a+|b|+1+2\varepsilon d\hat{\sigma}_k^2$. Note that 
		\begin{align*}
			\sup_{\varepsilon>0}(\varepsilon\e^{-(2a+|b|+1+2\varepsilon d\hat{\sigma}_k^2)T})=\lambda_T:=\frac{\e^{-(2a+|b|+1)T-1}}{2d\hat{\sigma}_k^2T}.
		\end{align*}
		Consequently,  by taking $\gamma\downarrow2a+|b|+1+\frac{1}{T}$ in \eqref{eq:Esup}, we arrive that
		\begin{equation}\label{eq:boundedfinite}
			\sup_{0\leq t \leq T}\mbb E\left(\e^{\lambda_0(|\widetilde{X}^{(k)}(t)|^2+|\widetilde{Y}^{(k)}(t)|^2)}\right)<\infty, \quad \lambda_0\in(0, \lambda_T).
		\end{equation}
		As a consequence, \eqref{eq:bounded} holds from \eqref{eq:boundedguji} and \eqref{eq:boundedfinite}.
		
		Let $\tau_R:=\inf\{t\geq0; |\widetilde{Y}^{(k)}(t)|=R\}$ where $ R\geq R_0:=|\varphi_2(0)|+1$. By $\textbf{(A3)}$, we have
		\begin{equation}\label{eq:aRt}
			a(R, t,\varphi):=\inf\left\{\underline{\lambda};|\psi_1(0)|\leq\left(|\varphi_1(0)|+|b|Rt\right)\e^{at}, |\psi_2(0)|\leq R\right\} >0,
		\end{equation}
		where $\underline{\lambda}$ is the lowest eigenvalue of $\sigma(\psi(0),k)$. Hence $\sigma(\widetilde{Z}^{(k)}(s), k)\geq a(R, T,\varphi)I$ for all $0\leq s\leq T\wedge\tau_R$. Note that for any $(\psi,k)\in\mcl C_r^{2d}\times\mbb S$ and $\varepsilon>0$, according to $\textbf{(A2)}$ and Young's inequality, we arrive at
		\begin{equation}\label{eq:sigmab2}
			\begin{aligned}
				|\sigma&^{-1}(\psi(0),k)b_2(\psi(0),k)|^2\\
				&\leq \|\sigma^{-1}(\psi(0),k)\|_{\rm HS}^{2}|b_2(\psi(0),k)|^2\\
				&\leq\|\sigma^{-1}(\psi(0),k)\|_{\rm HS}^{2}(|b_2(0, k)|+L_2|\psi(0)|^{\alpha})^2\\
				&\leq \|\sigma^{-1}(\psi(0),k)\|_{\rm HS}^{2}\left((1+\frac{1}{\varepsilon})|b_2(0, k)|^2+(1+\varepsilon)L_2^2(|\psi_1(0)|^{2\alpha}+|\psi_2(0)|^{2\alpha})\right)\\
				&\leq \|\sigma^{-1}(\psi(0),k)\|_{\rm HS}^{2}\left(C_{\varepsilon, \alpha, L_2}+\varepsilon(|\psi_1(0)|^2+|\psi_2(0)|^2)\right),   	
			\end{aligned}
		\end{equation}
		where $C_{\varepsilon, \alpha, L_2}$ is a positive constant. This yields that
		\begin{equation}\label{eq:guji}
			\begin{aligned}
				\mbb E&\exp{\left\{\lambda\int_{0}^{T\wedge\tau_R}|\sigma^{-1}(\widetilde{Z}^{(k)}(t),k)b_2(\widetilde{Z}^{(k)}(t), k)|^2\d t\right\}}\\
				&\leq \mbb E\exp{\left\{\frac{\lambda d}{a^2(R, T,\varphi)}\int_{0}^{T}\left(C_{\varepsilon, \alpha, L_2}+\varepsilon(|\widetilde{X}^{(k)}(t)|^2+|\widetilde{Y}^{(k)}(t)|^2)\right)\d t\right\}}.
			\end{aligned}
		\end{equation}
		Combining \eqref{eq:bounded} with \eqref{eq:guji} enables us to obtain that
		\begin{align*}
			\mbb E \exp{\left\{\lambda\int_{0}^{T\wedge\tau_R}|\sigma^{-1}(\widetilde{Z}^{(k)}(t),k)b_2(\widetilde{Z}^{(k)}(t), k)|^2\d t\right\}}<\infty,\  \text{for all $\lambda, T>0$.}
		\end{align*}
		
		
		In what follows, it remains to verify that $\tau_R\rightarrow\infty$ a.s. as $R\rightarrow\infty$. We shall show that for each $t>0$ fixed,
		\begin{equation}\label{eq:tauR}
			\lim\limits_{R\rightarrow+\infty}\mbb P(\tau_R\leq t)=0.
		\end{equation}
		To achieve this goal, let us consider a natural teat function: $H(x, y):=|x|^2+|y|^2$ for $x$, $y\in\mbb R^d$. Then, by $\textbf{(A1)}$, we can derive that for any $\psi\in\mcl C_r^{2d}$,
		\begin{align*}
			\widetilde{\mcl L}^{(k)}(\psi) H(\psi(0))&=2\langle\psi_1(0), a\psi_1(0)+b\psi_2(0)\rangle+2\langle b_1(\psi,k), \psi_2(0)\rangle+\text{tr}(\sigma(\psi(0),k)\sigma(\psi(0),k)^T)\\
			&\leq 2a|\psi_1(0)|^2+|b|(|\psi_1(0)|^2+|\psi_2(0)|^2)+|\psi_2(0)|^2+{C_{k, \sigma_k}}\\      	
			&\leq C_{k, \sigma_k}+(2a+|b|+1)|\psi(0)|^2.
		\end{align*}
		Now fix any $\kappa\geq 2a+|b|+1$. By It\^o's formula, for any $t\geq0$,
		\begin{equation*}\label{eq:operatortauR}
			\begin{aligned}
				\mbb E(\e^{-\kappa(t\wedge\tau_R)}H(\widetilde{Z}^{(k)}(t\wedge\tau_R))
				&=H(\varphi(0))+\mbb E\int_{0}^{t\wedge\tau_R}\e^{-\kappa s}\left(-\kappa H(\widetilde{Z}^{(k)}(s))+\widetilde{\mcl L}^{(k)}(\widetilde{Z}^{(k)}_s) H(\widetilde{Z}^{(k)}(s))\right)\d s\\
				&\leq H(\varphi(0))+\mbb E\int_{0}^{t\wedge\tau_R}\e^{-\kappa s}\left((2a+|b|+1-\kappa)|\widetilde{Z}^{(k)}(s)|^2+C_{k, \sigma_k}\right)\d s\\
				&\leq H(\varphi(0))+\frac{C_{k, \sigma_k}}{\kappa}.
			\end{aligned}
		\end{equation*} 
		which implies that 
		\begin{align*}
			R^2\mbb P(\tau_R\leq t)&\leq \mbb E\left[\mbf 1_{\{\tau_R\leq t\}}H(\widetilde{Z}^{(k)}(t\wedge\tau_R))\right]\leq \e^{\kappa t}\mbb E\left[ \e^{-\kappa(t\wedge\tau_R)}H(\widetilde{Z}^{(k)}(t\wedge\tau_R))\right]\leq \e^{\kappa t}\left(H(\varphi(0))+\frac{C_{k, \sigma_k}}{\kappa}\right).
		\end{align*} 
		This means that \eqref{eq:tauR} holds. We therefore complete the proof.      
	\end{proof}

Now we state the main result of this section.
\begin{theorem}\label{thm:weaksolution}
	Suppose that $\textbf{(A1)-(A3)}$ hold. Then,  there exists a unique  martingale solution  $\mathbb{P}_{k}^{(\varphi)}$ for the operator $\mcl L^{(k)}$ starting from $\varphi$, where $\mcl L^{(k)}$ is defined as follows: 
	\begin{align*}
		\begin{aligned}
			\mathcal{L}^{(k)}(\psi) f(\psi(0)):=&\frac{1}{2} \operatorname{tr}\left(a(\psi(0), k) \nabla_{y}^{2} f(\psi(0))\right)+\left\langle a\psi_1(0)+b\psi_2(0), \nabla_{x} f(\psi(0))\right\rangle \\
			&+\left\langle b_1(\psi, k)+b_2(\psi(0), k), \nabla_{y} f(\psi(0))\right\rangle,
		\end{aligned}
	\end{align*}
	for each $f\in\mcl C_c^{2}(\mbb R^d)$ and $\psi\in \mcl C_r^{2d}$.
\end{theorem}
\begin{proof}
	 First, we show the existence of the weak solution to \eqref{eq:zaikan}. For each initial value ${Z}^{(k)}_0=\varphi=:\varphi^1$, set
	\[
	\begin{aligned}
		R^{(k),1}(t):=\exp{\left\{\int_{0}^{t}\langle \sigma^{-1}(\widetilde{Z}^{(k)}(s),k)b_2(\widetilde{Z}^{(k)}(s),k), \d B(s)\rangle-\frac{1}{2}\int_{0}^{t}|\sigma^{-1}(\widetilde{Z}^{(k)}(s),k)b_2(\widetilde{Z}^{(k)}(s),k)|^2\d s\right\}},
	\end{aligned}
	\]
	for any $ t\in [0,T_1],$ and $\d \mathbb{P}_{k,T_1}^{(\varphi^1)}:=R^{(k),1}(T_1)\d \widetilde{\mbb P}^{(k)}$, where $T_1:=S>0$ is arbitrary with $\alpha\in(0,1)$. Moreover, let
	\begin{align*}
		\widetilde{B}^{(k),1}(t)=B(t)-\int_{0}^{t}\sigma^{-1}(\widetilde{Z}^{(k)}(s),k)b_2(\widetilde{Z}^{(k)}(s),k)\d s, \quad t\in[0,T_1].
	\end{align*}
	According to Lemma \ref{lem:Novikov}, we infer that
	\begin{align*}
		\mbb E_{\widetilde{\mbb P}^{(k)}}\exp{\left\{\frac{1}{2}\int_{0}^{T_1}|\sigma^{-1}(\widetilde{Z}^{(k)}(t)b_2(\widetilde{Z}^{(k)}(t)|^2\d t\right\}}<\infty,
	\end{align*}
	that is, the Novikov condition holds true. Thus the Girsanov theorem implies $(\widetilde{B}^{(k),1}(t))_{t\in[0, T_1]}$ is a Brownian motion under the probability measure $\mathbb{P}_{k,T_1}^{(\varphi^1)}$. Note that \eqref{eq:fuzhu} can be reformulated as
	\begin{align*}
		\left\{
		\begin{aligned}
			\text{d}\widetilde{X}^{(k)}(t)&=\left(a\widetilde{X}^{(k)}(t)+b\widetilde{Y}^{(k)}(t)\right)\text{d}t,\\
			\text{d}\widetilde{Y}^{(k)}(t)&=\left[ b_1(\widetilde{X}^{(k)}_t,\widetilde{Y}^{(k)}_t,k)+b_2(\widetilde{X}^{(k)}(t),\widetilde{Y}^{(k)}(t),k)\right]\text{d}t+\sigma(\widetilde{X}^{(k)}(t),\widetilde{Y}^{(k)}(t),k)\text{d}\widetilde{B}^{(k),1}(t),
		\end{aligned}
		\right.
	\end{align*}
	for $t\in[0, T_1)$ and $\widetilde{Z}^{(k)}_0=\varphi^1$.  Thus $(\widetilde{Z}^{(k),1}(t), \widetilde{B}^{(k),1}(t))$ is a weak solution to \eqref{eq:zaikan} on $[0, T_1)$ under the probability space $(\Omega, \mcl F, \mathbb{P}_{k,T_1}^{(\varphi^1)})$, where $\widetilde{Z}^{(k),1}(t):=\widetilde{Z}^{(k)}(t)$ for any $t\in[0,T_1).$ 
	
	Analogously, set $T_2:=2S$ and for $t\in[0,S]$, set $\widehat{Z}^{(k)}(t):=\widetilde{Z}^{(k)}(t+T_1)$. Moreover, for  initial value $\widehat{Z}^{(k)}_0=\widetilde{Z}^{(k),1}_{T_1}=:\varphi^2,$  set
	\[
	\begin{aligned}
		R^{(k),2}(t):=\exp{\left\{\int_{0}^{t}\langle \sigma^{-1}(\widehat{Z}^{(k)}(s),k)b_2(\widehat{Z}^{(k)}(s),k), \d B(s)\rangle-\frac{1}{2}\int_{0}^{t}|\sigma^{-1}(\widehat{Z}^{(k)}(s),k)b_2(\widehat{Z}^{(k)}(s),k)|^2\d s\right\}},
	\end{aligned}
	\]
	for any $ t\in [0,S],$ and $\d \mathbb{P}_{k,T_2}^{(\varphi^2)}:=R^{(k),2}(S)\d \widetilde{\mbb P}^{(k)}$.  Moreover, let
	\begin{align*}
		\widehat{B}^{(k)}(t)=B(t)-\int_{0}^{t}\sigma^{-1}(\widehat{Z}^{(k)}(s),k)b_2(\widehat{Z}^{(k)}(s),k)\d s, \quad t\in[0,S].
	\end{align*}
	According to Lemma \ref{lem:Novikov}, we infer that
	\begin{align*}
		\mbb E_{\widetilde{\mbb P}^{(k)}}\exp{\left\{\frac{1}{2}\int_{0}^{S}|\sigma^{-1}(\widehat{Z}^{(k)}(t)b_2(\widehat{Z}^{(k)}(t)|^2\d t\right\}}<\infty.
	\end{align*}
	Thus the Girsanov theorem implies $(\widehat{B}^{(k}(t))_{t\in[0, S]}$ is a Brownian motion under the probability measure $\mathbb{P}_{k,T_2}^{(\varphi^2)}$. We can reformulate  \eqref{eq:fuzhu} as
	\begin{align*}
		\left\{
		\begin{aligned}
			\text{d}\widehat{X}^{(k)}(t)&=\left(a\widehat{X}^{(k)}(t)+b\widehat{Y}^{(k)}(t)\right)\text{d}t,\\
			\text{d}\widehat{Y}^{(k)}(t)&=\left[ b_1(\widehat{X}^{(k)}_t,\widehat{Y}^{(k)}_t,k)+b_2(\widehat{X}^{(k)}(t),\widehat{Y}^{(k)}(t),k)\right]\text{d}t+\sigma(\widehat{X}^{(k)}(t),\widehat{Y}^{(k)}(t),k)\text{d}\widehat{B}^{(k)}(t),
		\end{aligned}
		\right.
	\end{align*}
	for $t\in[0, S)$ .  For $t\in[T_1,T_2)$, set
	\[
	(\widetilde{Z}^{(k),2}(t), \widetilde{B}^{(k),2}(t))=(\widehat{Z}^{(k),1}(t-T_1), \widetilde{B}^{(k),1}(T_1)+\widehat{B}^{(k)}(t-T_1)),
	\]
	which is a weak solution to \eqref{eq:zaikan} on $[T_1, T_2)$ under the probability space $(\Omega, \mcl F, \mathbb{P}_{k,T_2}^{(\varphi^2)})$. 
	Similarly,  we can show inductively that \eqref{eq:zaikan} admits a weak solution $(\widetilde{Z}^{(k),n}(t), \widetilde{B}^{(k),n}(t))$ on $[T_{n-1}, T_n)$ under the probability space $(\Omega, \mcl F, \mathbb{P}_{k, {T_n}}^{(\varphi^n)})$, $n\geq 3$, where  $\varphi^n=\widetilde{Z}^{(k),n-1}_{T_{n-1}}$ and $T_n:=nS$.
	
	Thus, we can define 
	\[
	(Z^{(k)}(t), B^{(k)}(t))=(\widetilde{Z}^{(k),n}(t), \widetilde{B}^{(k),n}(t)) \text{ for } t\in [T_{n-1}, T_n).
	\]
	For any initial value $\varphi\in\mcl C_r^{2d}$, we define a series of probability measures on $(\Omega, \mcl F)$ as follows:
	$$
	\mbb P^{(1)}=\mathbb{P}_{k,T_1}^{(\varphi^1)}, \quad {\rm and \  } \quad \mbb P^{(n+1)}=\mbb P^{(n)}\otimes_{T_{n}}\mathbb{P}_{k,{T_{n+1}}}^{(\varphi^{n+1})}, \quad \text{for $n\geq1$}.
	$$
	Note that $\lim\limits_{n\rightarrow 0}\mbb P^{(n)}(T_{n}\leq t)=0$ for all $t\geq0$. Hence by Tulcea's extension theorem (see, e.g., \cite[Theorem 1.3.5]{1997Multidimensional}), there exists a unique $\mathbb{P}_{k}^{(\varphi)}$ on $(\Omega, \mcl F)$ such that $\mathbb{P}_{k}^{(\varphi)}$ equals $\mbb P^{(n)}$ on $\mcl F_{T_n}:=\sigma(B^{(k)}(t):t\leq T_n)$.  Moreover, $B^{(k)}(t)$  is a Brownian motion under the probability measure $\mathbb{P}_{k}^{(\varphi)}$. Hence, \eqref{eq:zaikan} admits a global weak solution $(Z^{(k)}(t), B^{(k)}(t))$ under the probability space $(\Omega, \mcl F, \mathbb{P}_{k}^{(\varphi)})$.
	
	Next we proceed to show that the  probability measure $\mathbb{P}_{k}^{(\varphi)}$ is a martingale solution for   the operator $\mcl L^{(k)}$ starting from $\varphi$.   For each $f\in\mcl C_c^{2}(\mbb R^{2d})$, by It\^o's formula, we have
	\begin{align*}
		\begin{aligned}
			f(Z^{(k)}(t))=&f(\varphi(0))+\int_{0}^{t}\mcl L^{(k)}(Z^{(k)}_s)f(Z^{(k)}(s))\d s
			+\int_{0}^{t}\langle\nabla_{y} f(Z^{(k)}(s)),\sigma(Z^{(k)}(s),k)\d B^{(k)}(s)\rangle.
		\end{aligned}
	\end{align*}
	By $\textbf{(A3)}$,
	\begin{align*}
		\begin{aligned}
			M_t^{(k)(f)}&:=f(Z^{(k)}(t))-f(\varphi(0))-\int_{0}^{t}\mcl L^{(k)}(Z^{(k)}_s)f(Z^{(k)}(s))\d s
			=\int_{0}^{t}\langle\nabla_{y} f(Z^{(k)}(s)),\sigma(Z^{(k)}(s),k)\d B^{(k)}(s)\rangle
		\end{aligned}
	\end{align*}
	is an $\{\mcl F_t\}$-martingale with respect to $\mathbb{P}_{k}^{(\varphi)}$. 
	
	Finally, we justify the uniqueness of the martingale solution for the operator $\mcl L^{(k)}$. In the sequel, it is sufficient to show the weak uniqueness on the time interval $[0, T_1]$ since it can be done analogously on $[T_1, T_2]$, $[T_2, T_3]$, $\cdots$. Let $\mathbb{P}_{k}^{(\varphi),i}$, $i=1,2$,  be two martingale solutions for the operator $\mcl L^{(k)}$  starting from $\varphi$.  We intend to prove $\mathbb{P}_{k}^{(\varphi),1}=\mathbb{P}_{k}^{(\varphi),2}$.  In terms of \cite[Proposition 2.1, p169, \& Corollary, p206]{ikeda1989}, it remains to show that for any $f\in C_b(C([0, T_1]; \mbb R^{2d})\times C([0, T_1]; \mbb R^d); \mbb R)$,
	\begin{equation}\label{eq:unique}
		\mbb E_{\mathbb{P}_{k}^{(\varphi),1}}f(Z^{(k),1}(t), B^{(k)}_1(t))=\mbb E_{\mathbb{P}_{k}^{(\varphi),2}}f(Z^{(k),2}(t), B^{(k)}_2(t)),
	\end{equation}
	where $\mbb E_{\mathbb{P}_{k}^{(\varphi),i}}$ means the expectation w.r.t. $\mathbb{P}_{k}^{(\varphi),i}$. Whereas \eqref{eq:unique} can be done exactly by following the argument of \cite[Theorem 2.1 (2)]{wang2018}.  The proof is complete.
\end{proof}

\section{Martingale solution: special Markovian switching case}\label{sec:markov}
Before considering our desired results, we first construct a martingale solution to the operator $\widehat{\mcl A}$ defined in \eqref{eq:operatorhat}. To proceed, let us write $\omega=\left(\omega_1, \omega_2\right) \in \Omega:=\Omega_1 \times \Omega_2$ with $\Omega_1:=C([0, \infty), \mcl C_{r}^{2 d})$ and $\Omega_2:=D([0, \infty), \mathbb{S})$. We denote by $\mathcal{G}_t$ the $\sigma$-field generated by the cylindrical sets on $\Omega_1$ up to time $t$ and $\mathcal{N}_t$ the $\sigma$-field generated by the cylindrical sets on $D([0, \infty), \mathbb{S})$ up to time $t$. Put $\mathcal{G}:=\bigvee_{t=0}^{\infty} \mathcal{G}_t$ and $\mathcal{N}:=\bigvee_{t=0}^{\infty} \mathcal{N}_t$. We have $\mathcal{F}_t=\mathcal{G}_t \otimes \mathcal{N}_t$ for each $t \geq 0$ and $\mathcal{F}=\mathcal{G} \otimes \mathcal{N}$. 

Consider a special $Q$-matrix $\widehat{Q}=(\hat{q}_{kl})$ given by
\begin{equation}\label{eq:qhat}
	\hat{q}_{kl}:= \sup_{\varphi \in \mcl C_r^{2d}}q_{kl}(\varphi) \quad {\rm for} \quad l\neq k, \quad {\rm and} \quad \hat{q}_{kk}:=-\sum_{l\neq k}\hat{q}_{kl} \quad {\rm for} \quad k\in \mbb S.    
\end{equation}
As usual, denote by $\mathcal{B}_b(\mathbb{S})$ the Banach space of all bounded measurable functions on $\mathbb{S}$ equipped with the supremum norm. Corresponding to the $Q$-matrix $\widehat{Q}$, we introduce an operator $\widehat{Q}$ on $\mathcal{B}_b(\mathbb{S})$ as follows: for any $f \in \mathcal{B}_b(\mathbb{S})$,
\begin{equation}\label{eq:operatorspec}
	\widehat{Q} f(k)=\sum_{l \in \mathbb{S}} \hat{q}_{k l}(f(l)-f(k)) .
\end{equation}
For a given $k \in \mathbb{S}$, a probability measure $\mathbb{Q}^{(k)}$ on $D([0, \infty), \mathbb{S})$ is said to be a solution to the martingale problem for the operator $\widehat{Q}$ starting from $k$, if $\mathbb{Q}^{(k)}(\Lambda(0)=k)=1$ and for each function $f \in \mathcal{B}_b(\mathbb{S})$,
\begin{equation}
	N_t^{(f)}:=f(\Lambda(t))-f(\Lambda(0))-\int_0^t \widehat{Q} f(\Lambda(s)) \mathrm{d}s
\end{equation}
is an $\left\{\mathcal{N}_t\right\}$-martingale with respect to $\mathbb{Q}^{(k)}$. Here $\Lambda(t)$ is the coordinate process $\Lambda(t, \omega):=\omega(t)$ with $\omega \in D([0, \infty), \mathbb{S})$ and $t \geq 0$.

Now we introduce an operator $\widehat{\mathcal{A}}$ on $C_c^2(\mathbb R^{2d} \times \mathbb{S})$ as follows: for any $f\in C_{c}^{2}(\mbb{R}^{2 d} \times \mathbb{S})$ ,
\begin{equation}\label{eq:operatorhat}
	(\widehat{\mathcal{A}} f)(\varphi, k):=\mathbb{L}_{k}(\varphi)f(\varphi(0), k)+\widehat{Q} f(\varphi(0), k),
\end{equation}
where the operators $\mathbb{L}_k(\varphi)$ and $\widehat{Q}$ are defined in \eqref{eq:operatork} and \eqref{eq:operatorspec}, respectively. Note that $\widehat{\mathcal{A}}$ of \eqref{eq:operatorhat} is really a special case of the operator $\mathcal{A}$ defined in \eqref{eq:operator}. Similar to the notion of martingale solution for the operator $\mathcal{A}$ given in Definition \ref{defn1}, we give the definition of martingale solution for the operator $\widehat{\mathcal{A}}$.
\begin{definition}
	A probability measure $\widehat{\mathbb{P}}^{(\varphi, k)}$ on $C([0, \infty), \mcl C_{r}^{2d}) \times D([0, \infty)$, $\mathbb{S})$ is a solution to the martingale problem for the operator $\widehat{\mathcal{A}}$ starting from $(\varphi, k) \in \mcl C_{r}^{2 d} \times \mathbb{S}$, if $\widehat{\mathbb{P}}^{(\varphi, k)}\{(W_0, \Lambda(0))=(\varphi, k)\}=1$ and for each $f \in C_c^{2}(\mbb{R}^{2 d} \times \mathbb{S})$,
	\begin{equation}\label{eq:Mfhat}
		\widehat{M}_t^{(f)}(W_{\cdot}, \Lambda(\cdot)):=f(W(t), \Lambda(t))-f(W(0), \Lambda(0))-\int_0^t (\widehat{\mathcal{A}} f)(W_s, \Lambda(s)) \mathrm{d} s
	\end{equation}
	is a martingale with respect to the filtration $\left\{\mathcal{F}_t\right\}$ under $\widehat{\mathbb{P}}^{(\varphi, k)}$. Here, $(W_t, \Lambda(t))$ is the coordinate process on $C([0, \infty), \mcl C_{r}^{2 d}) \times D([0, \infty), \mathbb{S})$.
\end{definition}

We will show that for each $(\varphi, k) \in \mcl C_{r}^{2 d} \times \mathbb{S}$, there exists a unique martingale solution $\widehat{\mathbb{P}}^{(\varphi, k)}$ for the operator $\widehat{\mathcal{A}}$ starting from $(\varphi, k)$. From \cite{1986Zheng}, we have that for any given $k \in \mathbb{S}$, there exists a unique martingale solution $\mathbb{Q}^{(k)}$ on $D([0, \infty), \mathbb{S})$ for the operator $\widehat{Q}$ starting from $k$. Based on the martingale solutions $\left\{\mathbb{P}_k^{(\varphi)}: \varphi\in \mcl C_{r}^{2 d}, k \in \mathbb{S}\right\}$, $\left\{\mathbb{Q}^{(k)}: k \in \mathbb{S}\right\}$ and the stopping times $\left\{\tau_n\right\}$ defined in \eqref{eq:stopping}, we can construct the desired probability measure $\widehat{\mathbb{P}}^{(\varphi,k)}$ on $C([0, \infty), \mcl C_{r}^{2 d}) \times D([0, \infty), \mathbb{S})$. In order to accomplish the construction, the special matrix $\widehat{Q}$ being independent of $\varphi$ is very crucial; see Lemma \ref{lem:stopmart} and its proof below.

Let $(W_t, \Lambda(t))(\omega):=\left(\omega_1(t), \omega_2(t)\right)$ be the coordinate process on $\Omega$ and let $\left\{\tau_n\right\}$ be the sequence of stopping times defined by
\begin{equation}\label{eq:stopping}
	\tau_0\left(\omega_2\right) \equiv 0, \quad \tau_n\left(\omega_2\right):=\inf \left\{t>\tau_{n-1}\left(\omega_2\right): \Lambda\left(t, \omega_2\right) \neq \Lambda\left(\tau_{n-1}\left(\omega_2\right), \omega_2\right)\right\}\text { for any } n \geq 1. 
\end{equation}
According to $\textbf{(A4)}$, we have $\sup_{k\in\mbb S}\hat q_{k}\leq H<\infty$. Then it follows from \cite[Theorem 2.7.1]{1998Norris} that for any $k\in\mbb S$,
\begin{equation}\label{eq:q1}
	\mbb Q^{(k)}\left\{\lim\limits_{n\rightarrow\infty}\tau_n=+\infty\right\}=1.
\end{equation}

Next, we construct a Poisson random measure to rewrite the operator $\widehat{Q}$ as in \cite[Chapter II.9]{ikeda1989}. Precisely, for each $i \in \mbb S$, let
\[
\Delta_{i 1}=\left[0, \hat{q}_{i 1}\right), \Delta_{i j}=\left[\sum_{l=1,l\neq i}^{j-1}\hat{q}_{il},\sum_{j=1,j\neq i}^{j}\hat{q}_{ij}\right),j>1,j\neq i,
\]
and
\[
U_{i}=\bigcup_{j \geq 1, j \neq i} \Delta_{ij}, \quad i \geq 1.
\]
For notation convenience, we put \( \Delta_{i i}=\emptyset \) and \( \Delta_{i j}=\emptyset \) if \( \hat{q}_{i j}=0, i, j \in \mbb {S} \). Note that for each $i \in \mbb S$, $\{\Delta_{ij}:j\in \mbb S \}$ are disjoint intervals, and the length of the interval $\Delta_{ij}$ is equal to $\hat{q}_{ij}$. Moreover,  it follows from $\textbf{(A4)}$ that  $\mathfrak{m}\left(U_{i}\right)$ is bounded above by $H$  for any $i\in\S$,  where \( \mathfrak{m}(\mathrm{d} x) \) denotes the Lebesgue measure over \( \mathbb{R} \).

Let \( \xi_{n}^{(i)}, i, n=1,2, \ldots \), be \( U_{i} \)-valued random variables with
$
\mbb Q^{(k)}\left(\xi_{n}^{(i)} \in \mathrm{d} x\right)=\mathfrak{m}(\d x)/\mathfrak{m}\left(U_{i}\right),
$
and \( \tau_{n}^{(i)}, i, n \geq 1 \), be non-negative random variables satisfying \( \mbb Q^{(k)}\left(\tau_{n}^{(i)}>t\right)=\exp\{-\mathfrak{m}\left(U_{i}\right)t\} \), \( t \geq 0 \). Suppose that \( \left\{\xi_{n}^{(i)}, \tau_{n}^{(i)}\right\}_{i, n \geq 1} \) are all mutually independent. Put
\[
\zeta_{0}^{(i)}=0,\quad i \geq 1 \text {, and } \zeta_{n}^{(i)}=\tau_{1}^{(i)}+\cdots+\tau_{n}^{(i)} \quad\text { for }  i,n  \geq 1.
\]
Let
\[
D_{p}=\bigcup_{i \geq 1} \bigcup_{n \geq 0}\left\{\zeta_{n}^{(i)}\right\}\text{ and }p(\zeta_{n}^{(i)})=\xi_{n}^{(i)}\quad i,n  \geq 1.
\]
Correspondingly, put
\begin{equation}\label{eq:bujini}
	N([0, t] \times A)=\#\left\{ s \in D_{p}:0< s \leq t, p(s) \in A\right\}, t>0, A \in \mathscr{B}([0, \infty)) .
\end{equation}
As a consequence, we get a Poisson point process \( (p(t)) \) and a Poisson random measure \( {N}(\mathrm{d} t, \mathrm{d} u) \) with intensity \( \mathfrak{m}(\mathrm{d} u)\mathrm{d} t  \). Moreover, we know that \( \mathfrak{m}(\d u) \d t\) is the compensator of Poisson random measure $N(\d t, \d u)$; namely,
\begin{equation}\label{eq:nmart}
	\widetilde{N}(\d t,\d u):=N(\d t, \d u)-\mathfrak{m}(\d u)\d t
\end{equation}
is a martingale measure with respect to $\mathbb{Q}^{(k)}$. 

Define a function $h: \mathbb{S} \times\left[0, H\right]$ by
\[
h(i, u)=\sum_{j \in \mbb S}(j-i) \mathbf{1}_{\triangle_{i j}}(u).
\]
Thus, the operator $\widehat{Q}$ defined in \eqref{eq:operatorspec} can be represented as
\begin{equation}\label{eq:operatorqspec}
	\widehat{Q} f(i)=\sum_{j \in \mathbb{S}} \hat{q}_{i j}(f(j)-f(i))=\int_{[0,H]}(f(i+h(i,u))-f(i)) \mathfrak{m}(\d u) .
\end{equation}


\begin{lemma}\label{lem:stopmart}
	Under Assumptions $\textbf{(A1)-(A3)}$, for any $n \geq 1$, there exists a probability measure $\mathbb{P}_{\Lambda\left(\tau_n\right)}^{(W_{\tau_n})}$ on $\left(\Omega_1, \mathcal{G}\right)$ such that for any $f \in C_c^2(\mbb{R}^{2 d})$,
	$$
	f(W(t))-f\left(W\left(\tau_n\right)\right)-\int_{\tau_n}^t \mathbb{L}_{\Lambda\left(\tau_n\right)}(W_s) f(W(s)) \mathrm{d} s, \quad t \geq \tau_n
	$$
	is a martingale under $\mathbb{P}_{\Lambda\left(\tau_n\right)}^{(W_{\tau_n})}$.
\end{lemma} 
\begin{proof}
	Let us prove the lemma for the case when $n=1$. The proof for the general case is similar. For each $j \in \mathbb{S}$, it follows from Theorem \ref{thm:weaksolution} that for any $\varphi=(\varphi_1,\varphi_2) \in \mcl C_{r}^{2 d}$, the probability measure $\mathbb{P}_j^{(\varphi)}$ on $C([0, \infty), \mcl C_{r}^{2 d})$ is the unique solution to the martingale problem for the operator $\mathbb{L}_j$ starting from $\varphi$; that is, $\mathbb{P}_j^{(\varphi)}\left\{\omega_1: W_0(\omega_1)=\varphi\right\}=1$ and for each function $f \in C_c^{2}(\mathbb{R}^{2 d})$,
	\begin{equation}\label{eq:martk}
		f\left(W\left(t, \omega_1\right)\right)-f\left(W\left(0, \omega_1\right)\right)-\int_0^t \mathbb{L}_j(W_s) f\left(W(s)\left(\omega_1\right)\right) \mathrm{d} s, \quad t \geq 0
	\end{equation}
	is a $\left\{\mathcal{G}_t\right\}$-martingale with respect to $\mathbb{P}_j^{(\varphi)}$.   First, we shall show that for any $A \in \mathcal{G}$ and $k \in \mathbb{S}$, the function $\varphi \mapsto \mathbb{P}_k^{(\varphi)}(A)$ is measurable with respect to  the $\sigma$-field $\mathcal{G}$. It suffices to verify the property when the set $A$ is of the form:
	$$\left\{Z^{(k)}_{t_1} \in B_1, Z^{(k)}_{t_2} \in B_2,\dots, Z^{(k)}_{t_m}  \in B_m\right\},$$where $m \in \mathbb{N}, 0 \leq t_1<\cdots<t_m$, and $B_1, B_2, \ldots, B_m \in \mathcal{G}$.  Indeed, since
	$$
	\begin{aligned}
		& \mathbb{P}_k^{(\varphi)}\left\{Z^{(k)}_{t_1} \in B_1, Z^{(k)}_{t_2} \in B_2,\ldots, Z^{(k)}_{t_m}  \in B_m\right\} \\
		& \quad=\int_{B_1}\int_{B_2} \ldots \int_{B_m} P_k\left(t_1, \varphi, \varphi^1\right) P_k\left(t_2-t_1, \varphi^1, \varphi^2\right) \dots P_k\left(t_m-t_{m-1}, \varphi^{m-1}, \varphi^m\right) \mathrm{d} \varphi^m \mathrm{d} \varphi^{m-1} \ldots \mathrm{d} \varphi^1,
	\end{aligned}
	$$
	where $P_k(t, \varphi, \cdot):=\mathbb{P}_k^{(\varphi)}\{Z_t^{(k)}\in\cdot \} $ is the transition probability function of $Z_t^{(k)}$, the measurability is obvious. Thus, for $A \in \mathcal{G}$, we can define
	$$
	\mathbb{P}_j^{(W_{\tau_1\left(\omega_2\right)})}(A):=\int_{\mcl C_{r}^{2 d}} \mathbb{P}_j^{(\varphi^{\prime})}(A) \nu_1\left(\mathrm{d} \varphi^{\prime}\right),
	$$
	where $\nu_1\left(\cdot, \omega_2\right):=\mathbb{P}_k^{(\varphi)}\left\{\omega_1: W_{\tau_1\left(\omega_2\right)}\left(\omega_1\right) \in \cdot\right\}$ be the law of $W_{\tau_1\left(\omega_2\right)}$ under $\mathbb{P}_k^{(\varphi)}$ for each $\omega_2 \in D([0, \infty), \mathbb{S})$.  For each $j \in \mathbb{S}\backslash\{k\}$ and any $f \in C_c^2(\mathbb{R}^{2 d})$, we consider the following process
	$$
	\theta_j\left(t, \omega_1\right):=1_{\left\{\Lambda\left(\tau_1\left(\omega_2\right), \omega_2\right)=j\right\}}\left[f\left(W\left(t, \omega_1\right)\right)-f\left(W\left(\tau_1\left(\omega_2\right), \omega_1\right)\right)-\int_{\tau_1\left(\omega_2\right)}^t \mathbb{L}_j(W_s) f\left(W(s)\left(\omega_1\right)\right) \mathrm{d} s\right],
	$$
	$t \geq \tau_1\left(\omega_2\right)$. Recall the fact that $\widehat{Q}$ is independent of $\varphi$. Hence, for each $\omega_2 \in \Omega_2, \tau_1\left(\omega_2\right)$ is independent of $\omega_1$ and it can be regarded as a constant on $\left(\Omega_1, \mathcal{G}\right)$. In addition, for any $\tau_1\left(\omega_2\right) \leq t_1 \leq t_2$ and $A \in \mathcal{G}_{t_1}$, we have
	$$
	\begin{aligned}
		\int_A \theta_j\left(t_2,\omega_1\right) \mathbb{P}_j^{(W_{\tau_1(\omega_2)})}\left(\mathrm{d} \omega_1\right) & =\int_A  \theta_j\left(t_2,\omega_1\right) \int_{\mcl C_{r}^{2 d}}\mathbb{P}_j^{(\varphi^{\prime})}\left(\mathrm{d} \omega_1\right) \nu_1\left(\mathrm{d}\varphi^{\prime}\right)\\
		&=\int_{\mcl C_{r}^{2 d}} \int_A \theta_j\left(t_2,\omega_1\right) \mathbb{P}_j^{(\varphi^{\prime})}\left(\mathrm{d} \omega_1\right) \nu_1\left(\mathrm{d}\varphi^{\prime}\right) \\
		& =\int_{\mcl C_{r}^{2 d}} \int_A \theta_j\left(t_1,\omega_1\right) \mathbb{P}_j^{(\varphi^{\prime})}\left(\mathrm{d} \omega_1\right) \nu_1\left(\mathrm{d}\varphi^{\prime}\right)\\
		&=\int_A \theta_j\left(t_1,\omega_1\right) \mathbb{P}_j^{(W_{\tau_1(\omega_2)})}\left(\mathrm{d} \omega_1\right),
	\end{aligned}
	$$
	where we used \eqref{eq:martk} to obtain the third equality above. Thus, $\left\{\theta_j(t), t \geq \tau_1\left(\omega_2\right)\right\}$ is a martingale under $\mathbb{P}_j^{(W_{\tau_1})}$ with respect to $\left\{\mathcal{G}_t\right\}$.
	
	Next, let us define
	$$
	\mathbb{P}_{\Lambda\left(\tau_1\left(\omega_2\right)\right)}^{(W_{\tau_1\left(\omega_2\right)})}(A):=\sum_{l \neq k} 1_{\left\{\Lambda\left(\tau_1\left(\omega_2\right),\omega_2\right\}=l\right\}} \mathbb{P}_l^{(W_{\tau_1\left(\omega_2\right)})}(A), \quad A \in \mathcal{G} .
	$$
	Apparently for each $\omega_2 \in \Omega_2, \mathbb{P}_{\Lambda\left(\tau_1\left(\omega_2\right)\right)}^{(W_{\tau_1\left(\omega_2\right)})}$ is a probability measure on $\left(\Omega_1, \mathcal{G}\right)$. For simplicity, let us write $\mathbb{P}_{\Lambda\left(\tau_1\left(\omega_2\right)\right)}^{(W_{\tau_1\left(\omega_2\right)})}$ for $\mathbb{P}_{\Lambda\left(\tau_1\right)}^{(W_{\tau_1})}$ in the sequel. We need to show that for each $\omega_2 \in \Omega_2$,
	$$
	\theta\left(t, \omega_1\right):=f\left(W\left(t, \omega_1\right)\right)-f\left(W\left(\tau_1\left(\omega_2\right), \omega_1\right)\right)-\int_{\tau_1\left(\omega_2\right)}^t \mathbb{L}_{\Lambda\left(\tau\left(\omega_2\right), \omega_2\right)}(W_s) f\left(W(s)\left(\omega_1\right)\right) \mathrm{d} s, \quad t \geq \tau_1\left(\omega_2\right),
	$$
	is a martingale under $\mathbb{P}_{\Lambda\left(\tau_1\right)}^{(W_{\tau_1})}$. To this end, let $\tau_1\left(\omega_2\right) \leq t_1 \leq t_2$ and $A \in \mathcal{G}_{t_1}$ be given arbitrarily. Note that $\sum_{j \neq{k}} \mathbf{1}_{\left\{\Lambda\left(\tau_1\left(\omega_2\right), \omega_2\right)=j\right\}}=1$ and hence
	$$
	\theta\left(t, \omega_1\right)=\sum_{j \neq{k}} \theta\left(t, \omega_1\right) \mathbf{1}_{\left\{\Lambda\left(\tau_1\left(\omega_2\right) ,\omega_2\right)=j\right\}}=\sum_{j \neq{k}} \theta_j\left(t, \omega_1\right)
	$$
	for any $t \geq 0$. Therefore, we can compute
	$$
	\begin{aligned}
		\int_A \theta\left(t_2,\omega_1\right) \mathbb{P}_{\Lambda\left(\tau_1\right)}^{(W_{\tau_1})}\left(\mathrm{d} \omega_1\right)&=\int_A \sum_{j\neq k} \theta\left(t_2,\omega_1\right) 1_{\left\{\Lambda\left(\tau_1\right)=j\right\}} \sum_{l \neq k} 1_{\left\{\Lambda\left(\tau_1\right)=l\right\}} \mathbb{P}_l^{(W_{\tau_1})}\left(\mathrm{d} \omega_1\right) \\
		& =\int_A \sum_{j\neq k} \theta\left(t_2,\omega_1\right) 1_{\{\Lambda(\tau)=j\}} \mathbb{P}_j^{(W_{\tau_1})}\left(\mathrm{d} \omega_1\right)\\
		& =\sum_{j\neq k} \int_A \theta_j\left(t_2,\omega_1\right) \mathbb{P}_j^{(W_{\tau_1})}\left(\mathrm{d} \omega_1\right)\\
		&=\sum_{j\neq k} \int_A \theta_j\left(t_1,\omega_1\right) \mathbb{P}_j^{(W_{\tau_1})}\left(\mathrm{d} \omega_1\right)\\
		& =\int_A \sum_{j\neq k} \theta\left(t_1,\omega_1\right) \mathbf{1}_{\left\{\Lambda\left(\tau_1\right)=j\right\}} \sum_{l \neq k} \mathbf{1}_{\left\{\Lambda\left(\tau_1=l\right\}\right.} \mathbb{P}_l^{(W_{\tau_1})}\left(\mathrm{d} \omega_1\right)\\
		&=\int_A \theta\left(t_1,\omega_1\right) \mathbb{P}_{\Lambda\left(\tau_1\right)}^{(W_{\tau_1})}\left(\mathrm{d} \omega_1\right). \\
	\end{aligned}
	$$
	Hence, $\left\{\theta(t,\omega_1), t \geq \tau_1\left(\omega_2\right)\right\}$ is a martingale under $\mathbb{P}_{\Lambda\left(\tau_1\right)}^{(W_{\tau_1})}$ for each $\omega_2 \in \Omega_2$.
\end{proof}

Now we present the main result of this section.
\begin{theorem}\label{thm:markovsolution}
	Under Assumptions $\textbf{(A1)-(A4)}$, for any given $(\varphi, k) \in \mcl C_{r}^{2 d} \times \mathbb{S}$, there exists a unique martingale solution $\widehat{\mathbb{P}}^{(\varphi, k)}$ on $C([0, \infty), \mcl C_{r}^{2 d}) \times D([0, \infty), \mathbb{S})$ for the operator $\widehat{\mathcal{A}}$ starting from $(\varphi, k)$.
\end{theorem}
\begin{proof}
	For any given $(\varphi, k) \in \mcl C_{r}^{2 d} \times \mathbb{S}$, we define a series of probability measures on $(\Omega, \mathcal{F})=$ $\left(\Omega_1 \times \Omega_2, \mathcal{G} \otimes \mathcal{N}\right)$ as follows:
	\begin{equation}\label{eq:pmeasure}
		\mathbb{P}^{(1)}=\mathbb{P}_k^{(\varphi)} \times \mathbb{Q}^{(k)}, \quad \text { and for } n \geq 1, \quad \mathbb{P}^{(n+1)}=\mathbb{P}^{(n)} \otimes_{\tau_n} \left(\mathbb{P}_{\Lambda\left(\tau_n\right)}^{(W_{\tau_n})} \times \mathbb{Q}^{(\Lambda\left(\tau_n\right))}\right),
	\end{equation}
	where $\tau_n\left(\omega_1, \omega_2\right):=\tau_n\left(\omega_2\right)$ is the switching time defined in \eqref{eq:stopping}, $\mathbb{P}_{\Lambda\left(\tau_n\right)}^{\left(W\left(\tau_n\right)\right)}$ is the probability measure on $\left(\Omega_1, \mathcal{G}\right)$ as in Lemma \ref{lem:stopmart}, and $\mathbb{Q}^{\left(\Lambda\left(\tau_n\right)\right)}$ is the regular conditional probability distribution of $\mathbb{Q}^{(k)}$ with respect to $\mathcal{N}_{\tau_{\mathrm{n}}}$. Thanks to \cite[Theorem 6.1.2]{1997Multidimensional}, $\mathbb{P}^{(n+1)}=\mathbb{P}^{(n)}$ on $\mathcal{F}_{\tau_{\mathrm{n}}}$.
	Let $f \in C_c^2(\mathbb{R}^{2 d} \times \mathbb{S})$. We have
	$$
	f\left(W\left(\tau_1 \wedge t\right), k\right)-f(W(0), k)-\int_0^{\tau_1 \wedge t} \mathbb{L}_k(W_s) f(W(s), k) \mathrm{d} s
	$$
	is a martingale with respect to $\mathbb{P}_k^{(\varphi)}$ and hence $\mathbb{P}^{(1)}$. On the other hand, using \eqref{eq:operatorqspec}, we can write
	
	\begin{align}
		\int_0^{\tau_1 \wedge t} &\nonumber \widehat{Q} f(W(s), \Lambda(s)) \mathrm{d} s \\\nonumber
		&=\int_0^{\tau_1 \wedge t}\int_{[0,H]}\left[ f(W(s),\Lambda(s-)+h(\Lambda(s-),u))-f(W(s),\Lambda(s-))\right] \mathfrak{m}(\d u)\d s\\\nonumber
		&=-\int_0^{\tau_1 \wedge t}\int_{[0,H]}\left[ f(W(s),\Lambda(s-)+h(\Lambda(s-),u))-f(W(s),\Lambda(s-))\right] \left({N}(\d s,\d u)-\mathfrak{m}(\d u)\d s\right)\\\nonumber
		&\quad+\int_0^{\tau_1 \wedge t}\int_{[0,H]}\left[ f(W(s),\Lambda(s-)+h(\Lambda(s-),u))-f(W(s),\Lambda(s-))\right]{N}(\d s,\d u)\\\nonumber
		&=-\int_0^{\tau_1 \wedge t}\int_{[0,H]}\left[ f(W(s),\Lambda(s-)+h(\Lambda(s-),u))-f(W(s),\Lambda(s-))\right]\widetilde{N}(\d s,\d u)\\\nonumber
		&\quad+f\left(W\left(\tau_1 \wedge t\right), \Lambda\left(\tau_1 \wedge t\right)\right)-f\left(W\left(\tau_1 \wedge t\right), \Lambda\left(\tau_1 \wedge t-\right)\right)\nonumber.
	\end{align}
	Then, using this and the definitions of the operators $\widehat{\mathcal{A}}$, $\mathbb{L}_k$ and $\widehat{Q}$, we have
	\begin{align}\label{eq:mmart}
			\widehat{M}_{\tau_1 \wedge t}^{(f)}=\nonumber& f\left(W\left(\tau_1 \wedge t\right), \Lambda\left(\tau_1 \wedge t\right)\right)-f(W(0), \Lambda(0))-\int_0^{\tau_1 \wedge t} \widehat{\mathcal{A}}(W_s) f(W(s), \Lambda(s)) \mathrm{d} s \\\nonumber
			= & f\left(W\left(\tau_1 \wedge t\right), \Lambda\left(\tau_1 \wedge t\right)\right)-f\left(W\left(\tau_1 \wedge t\right), \Lambda(0)\right) \\\nonumber
			& +f\left(W\left(\tau_1 \wedge t\right), \Lambda(0)\right)-f(W(0), \Lambda(0))-\int_0^{\tau_1 \wedge t} \mathbb{L}_{\Lambda(0)}(W_s) f(W(s), \Lambda(0)) \mathrm{d} s \\\nonumber
			& +\int_0^{\tau_1 \wedge t} \mathbb{L}_{\Lambda(0)}(W_s) f(W(s), \Lambda(0)) \mathrm{d} s-\int_0^{\tau_1 \wedge t} \widehat{\mathcal{A}}(W_s) f(W(s), \Lambda(s)) \mathrm{d} s \\
			= & f\left(W\left(\tau_1 \wedge t\right), \Lambda\left(\tau_1 \wedge t\right)\right)-f\left(W\left(\tau_1 \wedge t\right), \Lambda(0)\right)-\int_0^{\tau_1 \wedge t} \widehat{Q} f(W(s), \Lambda(s)) \mathrm{d} s
			\\\nonumber
			& + f\left(W\left(\tau_1 \wedge t\right), \Lambda(0)\right)-f(W(0), \Lambda(0))-\int_0^{\tau_1 \wedge t} \mathbb{L}_{\Lambda(0)}(W_s) f(W(s), \Lambda(0)) \mathrm{d} s\\\nonumber
			= & \int_0^{\tau_1 \wedge t}\int_{[0,H]}\left[ f(W(s),\Lambda(s-)+h(\Lambda(s-),u))-f(W(s),\Lambda(s-))\right]\widetilde{N}(\d s,\d u)
			\\\nonumber
			& + f\left(W\left(\tau_1 \wedge t\right), \Lambda(0)\right)-f(W(0), \Lambda(0))-\int_0^{\tau_1 \wedge t} \mathbb{L}_{\Lambda(0)}(W_s) f(W(s), \Lambda(0)) \mathrm{d} s.\nonumber
			\end{align}
	Recall from \eqref{eq:nmart} that $\widetilde{N}$ is a martingale measure with respect to $\mathbb{Q}^{(k)}$ and hence $\mathbb{P}^{(1)}$. Thus it follows that $\widehat{M}_{\tau_1 \wedge}^{(f)}$. is a martingale with respect to $\mathbb{P}^{(1)}$.
	Next, thanks to Lemma \ref{lem:stopmart},
	$$
	f\left(W\left(\tau_2 \wedge t\right), \Lambda\left(\tau_1\right)\right)-f\left(W\left(\tau_1\right), \Lambda\left(\tau_1\right)\right)-\int_{\tau_1}^{\tau_2 \wedge t} \mathbb{L}_{\Lambda\left(\tau_1\right)}(W_s) f\left(W(s), \Lambda\left(\tau_1\right)\right) \mathrm{d} s, \quad t \geq \tau_1
	$$
	is a martingale with respect to $\mathbb{P}_{\Lambda\left(\tau_1\right)}^{(W_{ \tau_1})}$ and hence also $\mathbb{P}_{\Lambda\left(\tau_1\right)}^{(W_{ \tau_1})} \times \mathbb{Q}^{\left(\Lambda\left(\tau_1\right)\right)}$. Similar calculations as those in \eqref{eq:mmart} give that
	$$
	\begin{aligned}
		f&\left( W\left(\tau_2 \wedge t\right), \Lambda\left(\tau_2 \wedge t\right)\right)-f\left(W\left(\tau_1\right), \Lambda\left(\tau_1\right)\right)-\int_{\tau_1}^{\tau_2 \wedge t} \widehat{\mathcal{A}}(W_s) f(W(s), \Lambda(s)) \mathrm{d} s \\
		&= f\left(W\left(\tau_2 \wedge t\right), \Lambda\left(\tau_1\right)\right)-f\left(W\left(\tau_1\right), \Lambda\left(\tau_1\right)\right)-\int_{\tau_1}^{\tau_2 \wedge t} \mathbb{L}_{\Lambda\left(\tau_1\right)}(W_s) f\left(W(s), \Lambda\left(\tau_1\right)\right) \mathrm{d} s \\
		&\quad+\int_{\tau_1}^{\tau_2 \wedge t}\int_{[0,H]}[f(W(s), \Lambda(s-)+h(\Lambda(s-),u))-f(W(s), \Lambda(s-))] \widetilde{N}(\mathrm{d} s, \mathrm{d} u), \quad t \geq \tau_1 .
	\end{aligned}
	$$
	Since $\mathbb{Q}^{\left(\Lambda\left(\tau_1\right)\right)}$ is a regular conditional probability distribution of $\mathbb{Q}^{(k)}$ with respect to $\mathcal{N}_{\tau_1}$, it follows that for all $A\in \mathcal{B}([0,H])$, $\widetilde{N}([0,t],A), t \geq \tau_1$ is a martingale with respect to $\mathbb{Q}^{(\Lambda\left(\tau_1\right))}$. Consequently the expression in the last line of the previous sentence is a martingale with respect to $\mathbb{Q}^{(\Lambda\left(\tau_1\right))}$ and hence also $\mathbb{P}_{\Lambda(\tau_1)}^{(W_{\tau_1})} \times \mathbb{Q}^{\left(\Lambda\left(\tau_1\right)\right)}$. Then the left hand side of the previous sentence, which is equal to $\widehat{M}_{\tau_2 \wedge t}^{(f)}-\widehat{M}_{\tau_1 \wedge t}^{(f)}$, is a	martingale with respect to $\mathbb{P}_{\Lambda\left(\tau_1\right)}^{(W_{\tau_1})} \times \mathbb{Q}^{\left(\Lambda\left(\tau_1\right)\right)}$. Therefore in view of \cite[Theorem 6.1.2]{1997Multidimensional}, $\widehat{M}_{\tau_2 \wedge \cdot}^{(f)}$ is a martingale with respect to $\mbb P^{(2)}$. In a similar fashion, we can show that $\widehat{M}_{\tau_n \wedge \cdot}^{(f)}$ is a martingale with respect to $\mathbb{P}^{(n)}$ for any $n \geq 1$.
	
	Next we show that $\lim _{n \rightarrow \infty} \mathbb{P}^{(n)}\left\{\tau_n \leq t\right\}=0$ for any $t \geq 0$. To this end, we consider functions of the form $f(x, k)=g(k)$, where $g \in \mathcal{B}(\mathbb{S})$. Then,
	$$
	\widehat{M}_{\tau_n \wedge\cdot }^{(g)}=g(\Lambda(\tau_n \wedge\cdot))-g(\Lambda(0))-\int_0^{\tau_n \wedge\cdot} \widehat{Q} g(\Lambda(s)) \mathrm{d} s
	$$
	is a martingale with respect to $\mathbb{P}^{(n)}$. Moreover, for any $A \in \mathcal{N}$, we define $\widehat{\mathbb{Q}}(A):=\mathbb{P}^{(n)}\left\{\Omega_1 \times A\right\}$. Then $\widehat{M}_{\tau_n \wedge\cdot}^{(g)}$ is a martingale with respect to $\widehat{\mathbb{Q}}$. On the other hand, $\widehat{M}_t^{(g)}$ is a martingale with respect to $\mathbb{Q}^{(k)}$. Then, $\widehat{M}_{\tau_n \wedge\cdot}^{(g)}$ is a martingale with respect to $\mathbb{Q}^{(k)}$ as well. By the uniqueness for the martingale problem for $\widehat{Q}$, we have $\widehat{\mathbb{Q}}=\mathbb{Q}^{(k)}$. Therefore it follows from \eqref{eq:q1} that
	$$
	\mathbb{P}^{(n)}\left\{\tau_n \leq t\right\}=\widehat{\mathbb{Q}}\left\{\tau_n \leq t\right\}=\mathbb{Q}^{(k)}\left\{\tau_n \leq t\right\} \rightarrow 0 \text {, as } n \rightarrow \infty.
	$$
	Recall that the probabilities $\mathbb{P}^{(n)}$ constructed in \eqref{eq:pmeasure} satisfies $\mathbb{P}^{(n+1)}=\mathbb{P}^{(n)}$ on $\mathcal{F}_{\tau_n}$. Hence by Tulcea's extension theorem (see, e.g., \cite[Theorem 1.3.5]{1997Multidimensional}), there exists a unique $\widehat{\mathbb{P}}$ on $(\Omega, \mathcal{F})$ such that $\widehat{\mathbb{P}}$ equals $\mathbb{P}^{(n)}$ on $\mathcal{F}_{\tau_n}$. Thus it follows that $\widehat{M}_{\tau_n \wedge\cdot}^{(f)}$ is a martingale with respect to $\widehat{\mathbb{P}}$ for every $n\geq 1$. In addition, for any $t\geq 0$, we have
	\begin{equation}\label{eq:shuxi}
		\widehat{\mathbb{P}}\left\{\tau_n \leq t\right\}=\mathbb{P}^{(n)}\left\{\tau_n \leq t\right\} \rightarrow 0, \text { as } n \rightarrow \infty .
	\end{equation}
	Thus $\tau_n \rightarrow \infty$ a.s.$\widehat{\mbb P}$ and hence $\widehat{M}_t^{(f)}$ is a local martingale with respect to $\widehat{\mathbb{P}}$. Note that $f \in C_c^2(\mathbb{R}^{2 d} \times \mathbb{S})$. Thus, $\widehat{M}_t^{(f)}$ is a martingale with respect to $\widehat{\mathbb{P}}$. This establishes that $\widehat{\mbb P}$ is the desired martingale solution staring from $(\varphi, k)$ to the martingale problem for $\widehat{\mathcal{A}}$. When we wish to emphasize the initial data dependence $W_0=\varphi$ and $\Lambda(0)=k$, we write this martingale solution as $\widehat{\mathbb{P}}^{(\varphi, k)}$. This establishes the existence of a martingale solution for the operator $\widehat{\mathcal{A}}$. The proof of uniqueness is very similar to that in \cite{XI20184277} and we omit the details here for brevity.
\end{proof}

\section{Martingale solution: general state-dependent switching case}
In this section, we construct the martingale solution for the general case. Throughout the remainder of the section, $\widehat{\mbb P}^{(\varphi, k)}$ (or simply $\widehat{\mbb P}$ if there is no need to emphasize the initial condition) denotes the unique martingale solution to $\widehat{\mathcal{A}}$. The corresponding expectation is denoted by $\widehat{\mbb E}^{(\varphi, k)}$ or $\widehat{\mbb E}$.     
To proceed, for any given \( t \geq 0 \), we define a function \( M_{t} \) on the sample path space as follows:
\begin{equation}\label{eq:buyi}
	\begin{aligned}
		M_{t}(W_{\cdot}, \Lambda(\cdot)):= & \prod_{i=0}^{n(t)-1} \frac{q_{\Lambda\left(\tau_{i}\right) \Lambda\left(\tau_{i+1}\right)}\left(W_{\tau_{i+1}}\right)}{\hat{q}_{\Lambda\left(\tau_{i}\right) \Lambda\left(\tau_{i+1}\right)}}  \exp \left(-\sum_{i=0}^{n(t)} \int_{\tau_{i}}^{\tau_{i+1} \wedge t}\left[q_{\Lambda\left(\tau_{i}\right)}(W_s)-\hat{q}_{\Lambda\left(\tau_{i}\right)}\right] \mathrm{d} s\right),
	\end{aligned}
\end{equation}
where
\[
q_{k}(\varphi)=\sum_{l \neq{k}} q_{k l}(\varphi),\quad \hat{q}_{k}=\sum_{l \neq{k}} \hat{q}_{k l},\quad n(t)=\max \left\{i \in \mathbb{N}: \tau_{i} \leq t\right\},
\]
and \( \left\{\tau_{i}\right\} \) is the sequence of stopping times defined in \eqref{eq:stopping}. In case \( n(t)=0 \), we use the convention that \( \prod_{i=0}^{-1} a_{i}:=1 \).

\begin{lemma}\label{eq:ninini}
	Suppose that $\textbf{(A1)-(A4)}$ hold. Then \( (M_{t}, \mathcal{F}_{t}, \widehat{\mathbb{P}}) \) is a square-integrable martingale with $\widehat{\mbb E}[M_t]=1$ for all $t\geq0$.
\end{lemma}
\begin{proof}
	Since the proof is rather long, we split the proof into three steps.
	
	\noindent Step 1. Define 
	\[
	g(k,l,\varphi):=
	\begin{cases}
		\frac{q_{kl}(\varphi)}{\hat{q}_{kl}}\mbf 1_{\{\hat{q}_{kl}>0\}}, & \text{if $k\neq l$, $\varphi\in\mcl C_{r}^{2d}$,}\\
		0, & \text{if $k=l$, $\varphi\in\mcl C_{r}^{2d}$.}
	\end{cases}
	\]
	We observe that if \( q_{k l}(\varphi)>0 \) for all \( k \neq l \) and \( \varphi \in \mcl C_{r}^{2d} \), then
	\[
	\begin{aligned}
		\prod_{i=0}^{n(t)-1} \frac{q_{\Lambda\left(\tau_{i}\right) \Lambda\left(\tau_{i+1}\right)}\left(W_{\tau_{i+1}}\right)}{\hat{q}_{\Lambda\left(\tau_{i}\right) \Lambda\left(\tau_{i+1}\right)}}&=\exp \left\{\sum_{i=0}^{n(t)-1} \log \frac{q_{\Lambda\left(\tau_{i}\right) \Lambda\left(\tau_{i+1}\right)}\left(W_{\tau_{i+1}}\right)}{\hat{q}_{\Lambda\left(\tau_{i}\right) \Lambda\left(\tau_{i+1}\right)}}\right\}\\
		&=\exp \left\{\int_{0}^{t} \int_{[0,H]}\log g(\Lambda(s-), \Lambda(s-)+h(\Lambda(s-),u), W_s) N(\mathrm{d} s, \mathrm{d} u)\right\},
	\end{aligned}
	\]
	where \( N(t, A) \) is the Poisson random measure defined by \eqref{eq:bujini}. Then it follows from the definition of \( M \) in \eqref{eq:buyi} that
	\begin{equation}\label{eq:zhishu}
		M_{t}(W_{\cdot}, \Lambda(\cdot))=\exp \{A(t)\}
	\end{equation}
	where
	\[
	A(t):=\int_{0}^{t} \int_{[0,H]}\log g(\Lambda(s-), \Lambda(s-)+h(\Lambda(s-),u), W_s) N(\mathrm{d} s, \mathrm{d} u)-\int_{0}^{t}\left[q_{\Lambda(s)}(W_s)-\hat{q}_{\Lambda(s)}\right] \mathrm{d} s.
	\]
	Now we apply Itô's formula for jump processes (see, e.g., \cite[Chapter2, Theorem 5.1]{ikeda1989}) to the process \( M_{t} \) :
			\begin{align}\label{eq:mduds}
			M_{t}(W_{\cdot}, \Lambda(\cdot))-1\nonumber&=\e^{A(t)}-\e^{A(0)} \\
			&=\int_{0}^{t} \int_{[0,H]} \e^{A(s-)}[g(\Lambda(s-), \Lambda(s-)+h(\Lambda(s-),u), W_s)-1] N(\mathrm{d} s, \mathrm{d} u)\\\nonumber
			&\quad-\int_{0}^{t} \e^{A(s)}\left[q_{\Lambda(s)}(W(s))-\hat{q}_{\Lambda(s)}\right] \mathrm{d} s.
			\end{align}
	Note that
	\[
	\begin{aligned}
		\int_{0}^{t} \e^{A(s)}\left[q_{\Lambda(s)}(W_s)-\hat{q}_{\Lambda(s)}\right] \mathrm{d} s & =\int_{0}^{t} \e^{A(s)} \sum_{l \neq \Lambda(s-)}\left[q_{\Lambda(s-)l}(W_s)-\hat{q}_{\Lambda(s-)l}\right] \mathrm{d} s \\
		& =\int_{0}^{t} \e^{A(s)} \sum_{l \neq \Lambda(s-)}[g(\Lambda(s-), l, W_s)-1] \hat{q}_{\Lambda(s-) l}\mathrm{d} s \\
		& =\int_{0}^{t} \int_{[0,H]} \e^{A(s)}[g(\Lambda(s-), \Lambda(s-)+h(\Lambda(s-),u), W_s)-1]\mathfrak{m}(\mathrm{d} u) \mathrm{d} s.
	\end{aligned}
	\]
	Putting these observations into \eqref{eq:mduds} and using \eqref{eq:zhishu}, we obtain
	\[
	M_{t}(W_{\cdot}, \Lambda(\cdot))-1=\int_{0}^{t} \int_{[0,H]}M_{s-}(W_{\cdot}, \Lambda(\cdot))[g(\Lambda(s-), \Lambda(s-)+h(\Lambda(s-),u), W_s)-1]\widetilde{N}(\d s,\d u) ,
	\]
	where $\widetilde{N}(\d s,\d u)$ is defined by \eqref{eq:nmart} and also a martingale measure with respect to $\widehat{\P}$. In fact, by Theorem \ref{thm:markovsolution}, for each function $f\in \mcl B_{b}(\mbb S)$,
	\begin{equation}
		\begin{aligned}
			N_t^{(f)}=&f(\Lambda(t))-f(\Lambda(0))-\int_{0}^{t}(\widehat{\mcl A}f)(\Lambda(s))\d s\\
			=&f(\Lambda(t))-f(\Lambda(0))-\int_{0}^{t}\widehat{Q}f(\Lambda(s))\d s\\
			=&\sum_{i=0}^{n(t)-1}f(\Lambda(\tau_{i+1}\wedge t))-f(\Lambda(\tau_{i}\wedge t))-\int_{0}^{t}\sum_{l\neq \Lambda(s) } \hat{q}_{\Lambda(s) l}(f(l)-f(\Lambda(s)))\d s\\
			=&\int_{0}^{t}\int_{[0,H]}[f\left(\Lambda(s-)+h(\Lambda(s-),u)\right)-f(\Lambda(s-))]\left(N(\d s,\d u)
			-\mathfrak{m}(\d u)\d s\right)\\
			=&\int_{0}^{t}\int_{[0,H]}[f\left(\Lambda(s-)+h(\Lambda(s-),u)\right)-f(\Lambda(s-))]\widetilde{N}(\d s,\d u)
		\end{aligned}
	\end{equation}
	is an $\{\mcl F_t\}$-martingale with to $\widehat{\mathbb{P}}$. Then, by the proof of \cite[Lemma 2.4]{1985Shiga}, we conclude that $\widetilde{N}$ is a martingale measure with respect to $\widehat{\mathbb{P}}$.
	
	\noindent Step 2. In general, there may exist some \( i \neq j \) and \( \varphi \in \mcl C_{r}^{2d} \) such that \( q_{i j}(\varphi)=0 \). We define 
	$$ 
	q_{k l}^{\varepsilon}(\varphi):=q_{k l}(\varphi)+\frac{\varepsilon}{2^{l}}\text{ and }\hat{q}_{k l}^{\varepsilon}:=\sup_{\varphi \in \mcl C_r^{2d}}q_{kl}^{\varepsilon}(\varphi) +\frac{\varepsilon}{2^{l}}
	$$ 
	for all \( k, l \in \mathbb{S} \) with \( k \neq l \) and \( \varphi \in \mcl C_{r}^{2d} \). Also, we let 
	\[
	q_{k}^{\varepsilon}(\varphi):=\sum_{l \neq{k}} q_{k l}^{\varepsilon}(\varphi)=q_{k}(\varphi)+\left(1-\frac{1}{2^k}\right)\varepsilon \text{ and }\hat{q}_{k}^{\varepsilon}:=\sum_{l \neq{k}} \hat{q}_{k l}^{\varepsilon}=\hat{q}_{k}+\left(1-\frac{1}{2^k}\right)\varepsilon
	\] 
	for all \( k \in \mathbb{S} \) and \( \varphi \in \mcl C_{r}^{2d} \). Then as \( \varepsilon \downarrow 0 \), we have
	\[
	q_{k l}^{\varepsilon}(\varphi) \rightarrow q_{k l}(\varphi) \text { and } {q}_{k}^{\varepsilon}(\varphi) \rightarrow {q}_{k}(\varphi)
	\]
	uniformly with respect to \( \varphi \in \mcl C_{r}^{2d} \) for all \( l \neq k \in \mathbb{S} \). 
	Moreover,
	\[
	\hat{q}_{k l}^{\varepsilon} \rightarrow\hat{q}_{k l}\text { and } \hat{q}_{k}^{\varepsilon}\rightarrow \hat{q}_{k}
	\]
	for all \( l \neq k \in \mathbb{S} \).
	Next we define
	\[
	\begin{aligned}
		&M_{t}^{\varepsilon}(W_{\cdot}, \Lambda(\cdot))\\
		:=	&\exp \left\{\int_{0}^{t} \int_{[0,H]}\log g^{\varepsilon}(\Lambda(s-), \Lambda(s-)+h(\Lambda(s-),u), W_s) N(\mathrm{d} s, \mathrm{d} u)-\int_{0}^{t}\left[q_{\Lambda(s)}^{\varepsilon}(W_s)-\hat{q}_{\Lambda(s)}^{\varepsilon}\right] \mathrm{d} s \right\},
	\end{aligned}
	\]
	where
	\[
	g^{\varepsilon}(k,l,\varphi):=
	\begin{cases}
		\frac{q_{kl}^{\varepsilon}(\varphi)}{\hat{q}_{kl}^{\varepsilon}}\mbf 1_{\{\hat{q}_{kl}^{\varepsilon}>0\}}, & \text{if $k\neq l$, $\varphi\in\mcl C_{r}^{2d}$,}\\
		0, & \text{if $k=l$, $\varphi\in\mcl C_{r}^{2d}$.}
	\end{cases}
	\]
	Thanks to $\textbf{(A4)}$ and the bounded convergence theorem, we have \( M_{t}^{\varepsilon}(W_{\cdot}, \Lambda(\cdot)) \rightarrow \) \( M_{t}(W_{\cdot}, \Lambda(\cdot)) \) as \( \varepsilon \downarrow 0 \). Moreover, by Step 1, we have
	\[
	M_{t}^{\varepsilon}(W_{\cdot}, \Lambda(\cdot))-1=\int_{0}^{t} \int_{[0,H]}M_{s-}^{\varepsilon}(W_{\cdot}, \Lambda(\cdot))[g^{\varepsilon}(\Lambda(s-), \Lambda(s-)+h(\Lambda(s-),u), W_s)-1]\widetilde{N}(\d s,\d u).
	\]
	Now passing to the limit as \( \varepsilon \downarrow 0 \), it follows from the bounded convergence theorem that
	\[
	M_{t}(W_{\cdot}, \Lambda(\cdot))-1=\int_{0}^{t} \int_{[0,H]}M_{s-}(W_{\cdot}, \Lambda(\cdot))[g(\Lambda(s-), \Lambda(s-)+h(\Lambda(s-),u), W_s)-1]\widetilde{N}(\d s,\d u) .
	\]
	Step 3. For each $n \in \mathbb{N}$, let $T_n:=\inf \{t \geq 0:|M(t)|>n\}$. Apparently $M_{\cdot \wedge T_n}$ is a $\widehat{\mathbb{P}}$-martingale. Moreover, 
	$$
	M_{t \wedge T_n}=1+\int_{0}^{t}\int_{[0,H]} \mbf 1_{\left\{s \leq T_n\right\}} M_{s-}[g(\Lambda(s-), \Lambda(s-)+h(\Lambda(s-),u), W_s)-1] \widetilde{N}(\mathrm{d} s, \mathrm{d} u).
	$$
	It is easy to see that $|g(k,l,\varphi)|\leq1$ for all $k, l\in\mbb S$ and $\varphi\in\mcl C_r^{2d}$. Hence it follows that
	$$
	\begin{aligned}
		\widehat{\mathbb{E}}\left[M_{t \wedge T_n}^2\right] & \leq 2+2 \widehat{\mathbb{E}}\left|\int_{0}^{t}\int_{[0,H]} \mbf 1_{\left\{s \leq T_n\right\}} M_{s-}[g(\Lambda(s-), \Lambda(s-)+h(\Lambda(s-),u), W_s)-1] \widetilde{N}(\mathrm{d} s, \mathrm{d} u)\right|^2 \\
		&=2+2 \widehat{\mathbb{E}}\left[\int_{0}^{t}\int_{[0,H]}\mbf 1_{\left\{s \leq T_n\right\}} M_{s-}^2[g(\Lambda(s-),\Lambda(s-)+h(\Lambda(s-),u), W_s)-1]^2 \mathfrak{m}(\d u)\d s\right] \\
		&\leq 2+8 H \int_0^t \widehat{\mathbb{E}}\left[\mathbf{1}_{\left\{s \leq T_n\right\}} M_{s-}^2\right] \mathrm{d} s=2+8 H \int_0^t \widehat{\mathbb{E}}\left[M_{s \wedge T_n}^2\right] \mathrm{d} s.
	\end{aligned}
	$$
	Using Gronwall's inequality, we have 
	\begin{equation}\label{eq:squinte}
		\widehat{\mathbb{E}}\left[M_{t \wedge T_n}^2\right] \leq 2 \e^{8 H t}.
	\end{equation}
	Then it follows that for each $t \geq 0$ fixed, we have 
	$$
	\sup\limits_{n \in \mathbb{N}} \widehat{\mathbb{E}}\left[M_{t \wedge T_n}^2\right] \leq 2 \e^{8 H t}<\infty. 
	$$
	By the Vall\'ee de Poussion theorem (see, for example, \cite[Proposition A.2.2]{Ethier1986}), the sequence $\left\{M_{t \wedge T_n}, n \in \mathbb{N}\right\}$ is uniformly integrable.
	
	We claim that $T_{\infty}=\lim_{n \rightarrow \infty} T_n=\infty$ a.s. $\widehat{\mbb P}$. In fact, on the set $\left\{T_n \leq t\right\}$, we have $M_{t \wedge T_n}^2 \geq n^2$. Therefore we have 
	\[
	\widehat{\mathbb{P}}\left\{T_n \leq t\right\}\leq \frac{\widehat{\mathbb{E}}\left[M_{t \wedge T_n}^2\right]}{n^2}\leq \frac{2 \e^{8 H t}}{n^2}\rightarrow 0, \text { as } n \rightarrow \infty .
	\]
	Hence, by the uniform integrability of the sequence $\left\{M_{t \wedge T_n}, n \in \mathbb{N}\right\}$, we obtain that $M_{\cdot}$ is a $\widehat{\mathbb{P}}$-martingale. In addition, using Fatou's Lemma in \eqref{eq:squinte} gives us $\widehat{\mathbb{E}}[M_t^2] \leq 2 \e^{8 H t}<\infty$. This completes the proof.
\end{proof}

By the martingality of $M_t$ with respect to $\widehat{\mathbb{P}}$, we can construct another probability measure $\mathbb{P}$ on $\Omega=C([0, \infty), \mcl C_r^{2 d}) \times D([0, \infty), \mathbb{S})$ such that $\mathbb{P}$ is a martingale solution to the operator $\mathcal{A}$.

\begin{theorem}\label{thm:weaksolutionS}
	Assume that $\textbf{(A1)-(A4)}$ hold. Then for any given $(\varphi, k) \in \mcl C_r^{2 d} \times \mathbb{S}$, there exists a unique martingale solution $\mathbb{P}^{(\varphi, k)}$ on $\Omega$ for the operator $\mathcal{A}$ starting from $(\varphi, k)$. In other words, there exists a unique weak solution $\mathbb{P}^{(\varphi, k)}$ on $\Omega$ to the system \eqref{eq:zaikan1} and \eqref{eq:zaikan2} with initial data $(\varphi, k)$.
\end{theorem}
\begin{proof}
	Since the proof is rather long, we split the proof into three steps.
	
	\noindent Step 1. For each $t \geq 0$ and each $A \in \mathcal{F}_t$, define
	\begin{equation}\label{eq:RN}
		\mathbb{P}_t^{(\varphi, k)}(A)=\int_A M_t(W_{\cdot}, \Lambda(\cdot)) \mathrm{d} \widehat{\mathbb{P}}^{(\varphi, k)}.
	\end{equation}
	For any $0\leq s<t$ and $A\in\mcl F_s$, we have
	\[
	\mathbb{P}_t^{(\varphi, k)}(A)=\int_A M_t(W_{\cdot}, \Lambda(\cdot)) \mathrm{d} \widehat{\mathbb{P}}^{(\varphi, k)}=\int_A M_s(W_{\cdot}, \Lambda(\cdot)) \mathrm{d} \widehat{\mathbb{P}}^{(\varphi, k)}=\mathbb{P}_s^{(\varphi, k)}(A).
	\]
	Thus, $\left\{\mathbb{P}_t^{(\varphi, k)}\right\}_{t \geq 0}$ is a consistent family of probability measures. Thus by Tulcea's extension theorem (see, e.g., \cite[Theorem 1.3.5]{1997Multidimensional}), there exists a unique probability measure $\mathbb{P}^{(\varphi, k)}$ on $(\Omega, \mathcal{F})$ which coincides with $\mathbb{P}_t^{(\varphi, k)}$ on $\mathcal{F}_t$ for all $t \geq 0$. Moreover, we will prove that the \( \mathbb{P}^{(\varphi, k)} \) is the desired martingale solution for the operator \( \mathcal{A} \) staring from \( (\varphi, k) \). 
	
	If we can show that $M_{t}(W_{\cdot},\Lambda(\cdot))M_{t}^{(f)}(W_{\cdot},\Lambda(\cdot))$ is a $\widehat{\mathbb{P}}^{(\varphi, k)}$-martingale, then for any $0 \leq s<t$ and $A \in \mathcal{F}_s$, we have
	$$
	\begin{aligned}
		\int_A M_t^{(f)}(Z_{\cdot},\Theta(\cdot))\mathrm{d} \mathbb{P}^{(\varphi, k)}
		&=\int_A  M_t^{(f)}(W_{\cdot},\Lambda(\cdot)) M_t(W_{\cdot},\Lambda(\cdot))\mathrm{d} \widehat{\mathbb{P}}^{(\varphi, k)}\\
		&=\int_A  M_s^{(f)}(W_{\cdot},\Lambda(\cdot))M_s(W_{\cdot},\Lambda(\cdot)) \mathrm{d} \widehat{\mathbb{P}}^{(\varphi, k)}\\
		&=\int_A M_s^{(f)}(Z_{\cdot},\Theta(\cdot)) \mathrm{d} \mathbb{P}^{(\varphi, k)},
	\end{aligned}
	$$
	where the first and the third equalities hold true since $\mathbb{P}^{(\varphi, k)}$ coincides with the probability measure $\mathbb{P}_t^{(\varphi, k)}$ given in \eqref{eq:RN}, and the second equality follows from the martingale property of $\left(M_{t}(W_{\cdot},\Lambda(\cdot))M_{t}^{(f)}(W_{\cdot},\Lambda(\cdot)), \mathcal{F}_t, \widehat{\mathbb{P}}^{(\varphi, k)}\right)$. This shows that $\mathbb{P}^{(\varphi, k)}$ is a martingale solution for the operator $\mathcal{A}$ starting from $(\varphi, k)$.
	
	\noindent Step 2. In what follows, as in the proof of \cite[Theorem 3.6]{XI20184277}, we check that for each function \( f \in C_{c}^{2}(\mathbb{R}^{2d} \times \mathbb{S})\), \(\left(M_{t}(W_{\cdot},\Lambda(\cdot)) M_{t}^{(f)}(W_{\cdot},\Lambda(\cdot)), \mathcal{F}_{t}, \widehat{\mathbb{P}}\right) \) is a martingale. Simply write $M_t=M_{t}(W_{\cdot},\Lambda(\cdot))$ and $M_{t}^{(f)}=M_{t}^{(f)}(W_{\cdot},\Lambda(\cdot))$. Using integration by parts, we derive that
	\begin{equation}\label{eq:kaixin}
		\begin{aligned}
			M_{t} M_{t}^{(f)}= & \int_{0}^{t} M_{s-}^{(f)} \mathrm{d} M_{s}+\int_{0}^{t} M_{s-} \mathrm{d} \widehat{M}_{s}^{(f)} 
			+\int_{0}^{t} M_{s-}\left(\mathrm{d} M_{s}^{(f)}-\mathrm{d} \widehat{M}_{s}^{(f)}\right)\\&+\sum_{s \leq t}\left(M_{s}-M_{s-}\right)\left(M_{s}^{(f)}-M_{s-}^{(f)}\right),
		\end{aligned}
	\end{equation}
	where \( \widehat{M}_{t}^{(f)} \) is defined in \eqref{eq:Mfhat}. Using \eqref{eq:Mf} and \eqref{eq:buyi}, we can compute
	\begin{align}
		\nonumber\sum_{s \leq t}&  \left(M_{s}-M_{s-}\right)\left(M_{s}^{(f)}-M_{s-}^{(f)}\right) \\\nonumber
		&=\sum_{s \leq t}\left(M_{s}-M_{s-}\right)[f(W(s), \Lambda(s))-f(W(s), \Lambda(s-))] \\\nonumber
		& =\int_{0}^{t} \int_{[0,H]} M_{s-}\left(\frac{M_{s}}{M_{s-}}-1\right)[f(W(s), \Lambda(s-)+h(\Lambda(s-),u))-f(W(s), \Lambda(s-))] N(\mathrm{d} s, \mathrm{d} u) \\\nonumber
		& =\int_{0}^{t} \int_{[0,H]} M_{s-}[g(\Lambda(s-), \Lambda(s-)+h(\Lambda(s-),u), W_s)-1]\\\nonumber
		&\quad\times[f(W(s), \Lambda(s-)+h(\Lambda(s-),u))-f(W(s), \Lambda(s-))] N(\mathrm{d} s, \mathrm{d} u).\nonumber
	\end{align}
	On the other hand,
	\begin{align}
		\int_{0}^{t}\nonumber& M_{s-}\left(\mathrm{d} M_{s}^{(f)}-\mathrm{d} \widehat{M}_{s}^{(f)}\right) \\\nonumber
		&= -\int_{0}^{t} M_{s-} \sum_{l \in \S}\left[g(\Lambda(s-),l,W(s))-1\right][f(W(s), l)-f(W(s), \Lambda(s-))]\hat{q}_{\Lambda(s-) l} \mathrm{d} s \\\nonumber
		&= -\int_{0}^{t}\int_{[0,H]} M_{s-}[g(\Lambda(s-), \Lambda(s-)+h(\Lambda(s-),u), W_s)-1]\\\nonumber
		&\quad\times[f(W(s), \Lambda(s-)+h(\Lambda(s-),u))-f(W(s), \Lambda(s-))]\mathfrak{m}(\d u)\d s.\nonumber
	\end{align}
	Then upon plugging the above equation into \eqref{eq:kaixin}, it follows that
	$$
	\begin{aligned}
		M_t M_t^{(f)}= & \int_0^t M_{s-}^{(f)} \mathrm{d} M_s+\int_0^t M_{s-} \mathrm{d} \widehat{M}_s^{(f)}+\int_{0}^{t}\int_{[0,H]} M_{s-}[g(\Lambda(s-), \Lambda(s-)+h(\Lambda(s-),u), W_s)-1] \\
		& \times[f(W(s), \Lambda(s-)+h(\Lambda(s-),u))-f(W(s), \Lambda(s-))] \widetilde{N}(\mathrm{d} s, \mathrm{d} u).
	\end{aligned}
	$$
	Recall that we have shown respectively in Theorem \ref{thm:markovsolution} and Lemma \ref{eq:ninini} that  $\widehat{M}^{(f)}$  and $ M_t$ are martingales
	under the measure $\widehat{\mathbb{P}}^{(\varphi, k)}$. And also note that $\widetilde{N}$ is a martingale measure with respect to $\widehat{\mathbb{P}}^{(\varphi, k)}$. Thus, we conclude immediately that $M_{t} M_{t}^{(f)}$ is a martingale under $\widehat{\mathbb{P}}^{(\varphi, k)}$.
	
	\noindent Step 3. For the proof of uniqueness, we can use the same arguments as those in the proofs of \cite[Theorem 3.6]{XI20184277} to show that any martingale solution $\widetilde{\mathbb{P}}$ for the operator $\mathcal{A}$ starting from $(\varphi, k)$ must agree with $\mathbb{P}^{(\varphi, k)}$ on $\mathcal{F}_{\tau_n}, n \in \mathbb{N}$. 
	
	Then, define a family of probability measures $\mathbb{P}_n$ on $(\Omega, \mathcal{F})$ via $\mathbb{P}_n(A):=\mathbb{P}^{(\varphi, k)}(A)$ for $A \in \mathcal{F}_{\tau_n}$. Apparently we have $\mathbb{P}_{n+1}=\mathbb{P}_n$ on $\mathcal{F}_{\tau_n}$. By \eqref{eq:shuxi} and \eqref{eq:squinte}, we obtain that for any $t \geq 0$,
	$$
	\mathbb{P}_n\left\{\tau_n \leq t\right\}=\mathbb{P}_t^{(\varphi, k)}\left\{\tau_n \leq t\right\}=\int_{\Omega} 1_{\left\{\tau_n \leq t\right\}} M_t \mathrm{d} \widehat{\mathbb{P}}^{(\varphi, k)}\leq\left(\widehat{\mathbb{P}}^{(\varphi, k)}\left\{\tau_n \leq t\right\}\right)^2\left(\mathbb{E}^{\widehat{\mathbb{P}}^{(\varphi, k)}}\left[M_t^2\right]\right)^2 \rightarrow 0,
	$$
	as $n\rightarrow\infty$. Then in view of the Tulcea's extension theorem, the desired uniqueness follows. This completes the proof.
\end{proof}

\section{Feller Property}
Up to now, we have proved that the martingale problem for the operator $\mcl A$ defined in \eqref{eq:operator} is well-posed under $\textbf{(A1)-(A4)}$. Consequently, for any $(\varphi, k)$, there exists a unique probability measure $\mbb P^{(\varphi,k)}$ on $\Omega$ under which the coordinate process $(Z_t,\Theta(t))$ satisfies $P^{(\varphi,k)}\{(Z_0,\Theta(0))=(\varphi,k)\}=1$, and for any $f\in C_c^{\infty}(\mbb R^{2d}\times\mbb S)$, the process $M_t^{(f)}$ defined in \eqref{eq:Mf} is an $\{\mcl F_t\}$-martingale. In this section, we shall prove that  the process $(Z_t, \Theta(t))$ possesses the Feller property on the probability space $(\Omega, \mcl F, \mbb P^{(\varphi,k)})$  and $\{\mcl F_t\}_{t\geq 0}$ is the  natural filtration  generated by $Z_t.$ To this end, we impose the following assumptions:
\begin{enumerate} 
	\item [\textbf{(H1)}] $b_1(\varphi, k)$ is bounded with respect to $(\varphi, k)$, and there exists a positive constant $L_1>0$ such that
	{\begin{equation}
			|b_1(\varphi,k)-b_1(\psi,k)|\leq L_1\|\varphi-\psi\|_r, \quad \varphi,\psi\in\mcl C_r^{2d}, k\in\S.
	\end{equation}}
    \item [\textbf{(H2)}] For every $k\in\mathbb S$, $\sigma(\cdot, k)\in C_b^3(\mathbb R^d; \mathbb R^{d\times d})$. Let $a(x, k):=\sigma(x,k)\sigma(x,k)^T$ and satisfies
    $$\sup\limits_{x\in\mathbb R^d,k\in\mathbb S}\|a(x,k)^{-1}\|_{\rm{HS}}<\infty.$$
\end{enumerate}

 According to Theorem \ref{thm:markovsolution}, the martingale problem for the operator $\widehat{\mcl A}$ defined in \eqref{eq:operatorhat} is well-posed. Therefore, there exists a probability measure $\widehat{\mbb P}$ under which the coordinate process $W(t):=(U(t), V(t))$ satisfies
\begin{equation}\label{eq:leilei}
	\left\{
	\begin{aligned}
		\text{d}U(t)&=\left(aU(t)+bV(t)\right)\text{d}t,\\
		\text{d}V(t)&=\left[ b_1(U_t,V_t,\Lambda(t))+b_2(U(t),V(t),\Lambda(t))\right]\text{d}t+\sigma(U(t),V(t),\Lambda(t))\text{d}B(t),
	\end{aligned}
	\right.
\end{equation}
and the component $\Lambda(t)$ is a time-homogeneous Markov chain with state space $\mbb S$ satisfying
\begin{equation}
	\mathbb{P}\{\Lambda(t+\Delta)=l \mid \Lambda(t)=k\}=\left\{\begin{array}{ll}
		\hat{q}_{k l} \Delta+o(\Delta), & \text { if } l \neq k, \\
		1+\hat{q}_{k k} \Delta+o(\Delta), & \text { if } l=k,
	\end{array}\right.
\end{equation}
provided $\Delta\downarrow0$, where $\widehat{Q}=(\hat{q}_{kl})$ is a conservation $Q$-matrix defined in \eqref{eq:qhat}. To proceed, we first give some necessary estimates.

\begin{lemma}\label{strlem:basic}
	Suppose that Assumptions $\textbf{(A2)-(A4)}$ and $\textbf{(H1)-(H2)}$ hold. Then, for all $T$, $\delta>0$, $\varphi, \psi\in\mcl C_r^{2d}$ and $k\in \mbb S$, we have
	\begin{equation}
		\mbb P\left\{\sup\limits_{0\leq t\leq T}\|W_t^{(\varphi,k)}-W_t^{(\psi,k)}\|_r\geq\delta\right\}\rightarrow0, \quad \text{as  $\|\varphi-\psi\|_r\rightarrow0$},
	\end{equation}
	where $W_t^{(\varphi, k)}=(U_t^{(\varphi, k)}, V_t^{(\varphi, k)})$.
\end{lemma}
\begin{proof}
	It is easy to see that
	\begin{equation}\label{eq:rnorm}
		\begin{aligned}
			\sup\limits_{0\leq t\leq T}\|W_t^{(\varphi,k)}-W_t^{(\psi,k)}\|_r&=\sup\limits_{0\leq t\leq T}\sup\limits_{-\infty<\theta\leq0}\e^{r\theta}\left|W^{(\varphi,k)}(t+\theta)-W^{(\psi,k)}(t+\theta)\right|\\
			&\leq\sup\limits_{0\leq t\leq T}\sup\limits_{-\infty<u\leq t}\e^{r(u-t)}\sqrt{\left|U^{(\varphi,k)}(u)-U^{(\psi,k)}(u)\right|^2+\left|V^{(\varphi,k)}(u)-V^{(\psi,k)}(u)\right|^2}\\
			&\leq\|\varphi-\psi\|_r+\sup\limits_{0\leq t\leq T}\left(\left|U^{(\varphi,k)}(t)-U^{(\psi,k)}(t)\right|+\left|V^{(\varphi,k)}(t)-V^{(\psi,k)}(t)\right|\right)\\
			&=:\|\varphi-\psi\|_r+\Gamma(T).
		\end{aligned}
	\end{equation}
	Now we directly compute $\Gamma(T)$ as follows:
	\[
	\begin{aligned}
		\Gamma(T)&=\sup\limits_{0\leq t\leq T}\left(\left|U^{(\varphi,k)}(t)-U^{(\psi,k)}(t)\right|+\left|V^{(\varphi,k)}(t)-V^{(\psi,k)}(t)\right|\right)\\
		&\leq 2\|\varphi-\psi\|_r+\int_0^T\left(a|U^{(\varphi,k)}(s)-U^{(\psi,k)}(s)|+|b||V^{(\varphi,k)}(s)-V^{(\psi,k)}(s)|\right)\d s\\
		&\quad+\int_0^T\left|b_1(W^{(\varphi,k)}_s,\Lambda^{(k)}(s))-b_1(W^{(\psi,k)}_s,\Lambda^{(k)}(s))\right|\d s\\
		&\quad+\int_0^T\left|b_2(W^{(\varphi,k)}(s),\Lambda^{(k)}(s))-b_2(W^{(\psi,k)}(s),\Lambda^{(k)}(s))\right|\d s\\
		&\quad+\sup\limits_{0\leq t\leq T}\left|\int_0^t\left(\sigma(W^{(\varphi,k)}(s),\Lambda^{(k)}(s))-\sigma(W^{(\psi,k)}(s),\Lambda^{(k)}(s))\right)\d B(s)\right|.
	\end{aligned}
	\]
	From the Markov  inequality and\textbf{ (A2)}, we have that for any given $\varepsilon>0$,
		\begin{align}
			\mathbb{P}\nonumber&\left\{\int_0^T\left|b_2(W^{(\varphi,k)}(s),\Lambda^{(k)}(s))-b_2(W^{(\psi,k)}(s),\Lambda^{(k)}(s))\right|\d s\geq \varepsilon \right\}\\\nonumber
			&\leq \frac{1}{\varepsilon}\mathbb{E}\int_0^T\left|b_2(W^{(\varphi,k)}(s), \Lambda^{(k)}(s))-b_2(W^{(\psi,k)}(s),\Lambda^{(k)}(s))\right|\d s\\\nonumber
			&\leq \frac{L_2}{\varepsilon}\mathbb{E}\int_0^T|W^{(\varphi,k)}(s)-W^{(\psi,k)}(s)|^{\alpha}\d s\\\nonumber
			&\leq \frac{L_2}{\varepsilon}\int_0^T\mathbb{E}\|W^{(\varphi,k)}_s-W^{(\psi,k)}_s\|_r^{\alpha}\d s\\\nonumber
			&\leq \frac{L_2}{\varepsilon}\int_0^T(\mathbb{E}\|W^{(\varphi,k)}_s-W^{(\psi,k)}_s\|_r)^{\alpha}\d s.
			\end{align}
	Hence, using the similar argument in \cite{xi2009}, we only need to declare that there exists a positive constant $C$ such that for any $\varphi,\ \psi\in\mcl C_r^{2d},\ k\in\mbb S$, we have 
	\begin{equation}\label{eq:mudi}
		\mbb E\|W_t^{(\varphi,k)}-W_t^{(\psi,k)}\|_r\leq C\|\varphi-\psi\|_r.
	\end{equation}
	
	On account of the impact of the H\"older continuity of $b_2$, we shall apply the transformation introduced by Flandoli-Gubinelli-Priola in \cite{Flandoli2010} to remove it. Let  $z=(x,y)$. For each $i\in\mbb S$ and $\lambda>0$, consider the vector valued elliptic equation
	\begin{equation}\label{eq:elliptic}
		\lambda u(z)-\mcl L^i u(z)=b_2(z,i),
	\end{equation}
	where 
	\[
	\mcl L^i u(z)=\frac{1}{2}{\rm tr}\left({\nabla^2_{y}}u(z)\sigma(z,i)\sigma(z,i)^T\right)+\langle\nabla_{y}u(z), b_2(z,i)\rangle,
	\]
	for any $u\in C^2(\mbb R^{2d}, \mbb R^{d})$. According to $\textbf{(H2)}$  and Theorem 5 in \cite{Flandoli2010}, the solution of equation \eqref{eq:elliptic} gives us a vector valued function $f_{\lambda}(\cdot,i):\mbb R^{2d}\rightarrow\mbb R^{d}$ and satisfies
	\begin{equation}\label{eq:supbounded}
		\sup\limits_{z\in\mbb R^{2d}}\left(|\nabla_zf_{\lambda}(z,i)|+\|\nabla^2_{z}f_{\lambda}(z,i)\|_{\text{ HS}}\right)<\frac{1}{2},
	\end{equation}
	for $\lambda$ large enough, where $\nabla_z $  and  \( \nabla_z^{2} \)   represents the gradient  and  the Hessian matrix of $f$ with respect to the variable $z$ respectively. Thus, the map
	\[
	g_{\lambda}(z,i)=y+f_{\lambda}(z,i)
	\]
	is a $C^2$-diffeomorphism of $\mbb R^{2d}$ with respect to $z$.  Let 
	\[
	\left\{
	\begin{aligned}
		\widetilde{U}(t)&=U(t),\\
		\widetilde{V}(t)&=g_{\lambda}(U(t),V(t),\Lambda(t))=V(t)+f_{\lambda}(U(t),V(t),\Lambda(t)),
	\end{aligned}
	\right.
	\]
	for $t\geq0$. Applying It\^o's formula and \eqref{eq:elliptic}, we obtain that
	\begin{equation}\label{eq:tildeU}
		\text{d}\widetilde{U}(t)=(aU(t)+bV(t))\d t,
	\end{equation}
	and
	\begin{equation}\label{eq:tildeV}
		\begin{aligned}
			\text{d}\widetilde{V}(t)=&\Big[\lambda f_{\lambda}(U(t), V(t),{\Lambda(t)})+b_1(U_t, V_t, \Lambda(t))+\langle\nabla_{x}f_{\lambda}(U(t), V(t),{\Lambda(t)}), aU(t)+bV(t)\rangle\\
			&+\langle\nabla_{y}f_{\lambda}(U(t), V(t),{\Lambda(t)}), b_1(U_t, V_t, \Lambda(t))\rangle\Big]\d t+\sum\limits_{j\in\mbb S}\hat{q}_{\Lambda(t)j}g_{\lambda}(U(t), V(t), j)\d t\\
			&+\langle\nabla_{y}g_{\lambda}(U(t), V(t),{\Lambda(t)}), \sigma(U(t), V(t), \Lambda(t))\d B(t)\rangle\\
			&+\int_{[0,H]}\left(g_{\lambda}(U(t), V(t), \Lambda(t-)+h(\Lambda(t-),z))-g_{\lambda}(U(t), V(t), \Lambda(t-)\right)\widetilde{N}(\d t, \d z).
		\end{aligned}
	\end{equation}
	Thus, by \eqref{eq:supbounded}, $\textbf{(H1)}$ and $\textbf{(H2)}$, the coefficients in \eqref{eq:tildeU}-\eqref{eq:tildeV} are Lipschitz continuous. In fact, for any ${\varphi},{\psi}\in\mcl C_r^{2d},\ k\in\mbb S$,  we have
	\begin{align*}
		|\langle\nabla&_{y}f_{\lambda}(\varphi(0),k), b_1(\varphi, k)\rangle-\langle\nabla_{y}f_{\lambda}(\psi(0),k), b_1(\psi, k)\rangle|\\
		&\leq |\langle\nabla_{y}f_{\lambda}(\varphi(0),k), b_1(\varphi, k)-b_1(\psi, k)\rangle|+|\langle\nabla_{y}f_{\lambda}(\varphi(0),k)-\nabla_{y}f_{\lambda}(\psi(0),k), b_1(\psi, k)\rangle|\\
		&\leq \frac{1}{2}| b_1(\varphi, k)-b_1(\psi, k)|+{ C}|\nabla_{y}f_{\lambda}(\varphi(0),k)-\nabla_{y}f_{\lambda}(\psi(0),k)|\\
		&\leq C\|\varphi-\psi\|_r,
	\end{align*}
	where $C$ is a positive constant. And the other terms can be treated similarly. Therefore, by a standard argument one can show that for any $\widetilde{\varphi},\widetilde{\psi}\in\mcl C_r^{2d},\ k\in\mbb S$ , there exists a positive constant  $\widetilde C$ such that
	\[
	\mbb E|\widetilde{W}^{(\widetilde{\varphi},k)}(t)-\widetilde{W}^{(\widetilde{\psi},k)}(t)|^2\leq \widetilde{C}\|\widetilde{\varphi}-\widetilde{\psi}\|_r^2,
	\]
	where $\widetilde{W}^{(\widetilde{\varphi},k)}(t)=(\widetilde{U}^{(\widetilde{\varphi},k)}(t),\widetilde{V}^{(\widetilde{\varphi},k)}(t))$.  For ${\varphi},{\psi}\in\mcl C_r^{2d}$ and $k\in\mbb S$, define $\widetilde{\varphi}(\theta):=(\varphi_1(\theta),g_{\lambda}(\varphi(\theta),k))$ and   $\widetilde{\psi}(\theta):=(\psi_1(\theta),g_{\lambda}(\psi(\theta),k))$ for any $\theta\in(\infty,0 ]$.  Hence, for some $C>0$,
	\begin{equation}\label{eq:a1ll}
		\begin{aligned}
			\|\widetilde{\varphi}-\widetilde{\psi}\|_r^2&= \sup\limits_{-\infty<\theta\leq0}\e^{2r\theta}(|\varphi_1(\theta)-\psi_1(\theta)|^2+|g_{\lambda}(\varphi(\theta),k)-g_{\lambda}(\psi(\theta),k)|^2)
			\leq C\|\varphi-\psi\|_r^2.
		\end{aligned}
	\end{equation}
	The definition of $\widetilde{W}(t)$ yields that
	\begin{align}\label{eq:jieshu}
		\mbb E\nonumber&|\widetilde{W}^{(\widetilde{\varphi},k)}(t)-\widetilde{W}^{(\widetilde{\psi},k)}(t)|^2\\\nonumber
		&= \mbb E\big[|U^{(\varphi,k)}(t)-U^{(\psi,k)}(t)|^2+|V^{(\varphi,k)}(t)-V^{(\psi,k)}(t)|^2+|f_{\lambda}(W^{(\varphi,k)}(t),{\Lambda^{(k)}(t)})-f_{\lambda}(W^{(\psi,k)}(t),{\Lambda^{(k)}(t)})|^2\\\nonumber
		&\quad+2\langle V^{(\varphi,k)}(t)-V^{(\psi,k)}(t),f_{\lambda}(W^{(\varphi,k)}(t),{\Lambda^{(k)}(t)})-f_{\lambda}(W^{(\psi,k)}(t),{\Lambda^{(k)}(t)})\rangle\big]\\\nonumber
		&\geq \mbb E\big[|W^{(\varphi,k)}(t)-W^{(\psi,k)}(t)|^2+|f_{\lambda}(W^{(\varphi,k)}(t),{\Lambda^{(k)}(t)})-f_{\lambda}(W^{(\psi,k)}(t),{\Lambda^{(k)}(t)})|^2\\
		&\quad-2|V^{(\varphi,k)}(t)-V^{(\psi,k)}(t)||f_{\lambda}(W^{(\varphi,k)}(t),{\Lambda^{(k)}(t)})-f_{\lambda}(W^{(\psi,k)}(t),{\Lambda^{(k)}(t)})|\big]\\\nonumber
		&\geq \mbb E\big[|W^{(\varphi,k)}(t)-W^{(\psi,k)}(t)|^2+|f_{\lambda}(W^{(\varphi,k)}(t),{\Lambda^{(k)}(t)})-f_{\lambda}(W^{(\psi,k)}(t),{\Lambda^{(k)}(t)})|^2\\\nonumber
		&\quad-2|W^{(\varphi,k)}(t)-W^{(\psi,k)}(t)||f_{\lambda}(W^{(\varphi,k)}(t),{\Lambda^{(k)}(t)})-f_{\lambda}(W^{(\psi,k)}(t),{\Lambda^{(k)}(t)})|\big]\\\nonumber
		&=\mbb E\left[|W^{(\varphi,k)}(t)-W^{(\psi,k)}(t)|-|f_{\lambda}(W^{(\varphi,k)}(t),{\Lambda^{(k)}(t)})-f_{\lambda}(W^{(\psi,k)}(t),{\Lambda^{(k)}(t)})|\right]^2.\nonumber
	\end{align}
	It follows from \eqref{eq:supbounded} that 
	\begin{equation}\label{eq:xiawu}
		|f_{\lambda}(W^{(\varphi,k)}(t),{\Lambda^{(k)}(t)})-f_{\lambda}(W^{(\psi,k)}(t),{\Lambda^{(k)}(t)})|\leq \frac{1}{2}|W^{(\varphi,k)}(t)-W^{(\psi,k)}(t)|
	\end{equation}
	for $\lambda$ large enough. Inserting \eqref{eq:xiawu} into \eqref{eq:jieshu}, we derive that
	\begin{equation*}
		\mbb E|\widetilde{W}^{(\widetilde{\varphi},k)}(t)-\widetilde{W}^{(\widetilde{\psi},k)}(t)|^2\geq \frac{1}{4}\mbb E|W^{(\varphi,k)}(t)-W^{(\psi,k)}(t)|^2.
	\end{equation*}
	Consequently,  analogous to the calculation of \eqref{eq:rnorm}, we obtain
	\[
	\begin{aligned}
		\mbb E\|W_t^{(\varphi,k)}-W_t^{(\psi,k)}\|_r&\leq\|\varphi-\psi\|_r+\sup\limits_{0<u\leq t}\mbb E|W^{(\varphi,k)}(u)-W^{(\psi,k)}(u)|\\
		&\leq\|\varphi-\psi\|_r+\sup\limits_{0<u\leq t}\left(\mbb E|W^{(\varphi,k)}(u)-W^{(\psi,k)}(u)|^2\right)^{\frac{1}{2}}\\
		&\leq \|\varphi-\psi\|_r+2\sup\limits_{0<u\leq t}\left(\mbb E|\widetilde{W}^{(\widetilde{\varphi},k)}(t)-\widetilde{W}^{(\widetilde{\psi},k)}(t)|^2\right)^{\frac{1}{2}}\\
		&\leq C\|\varphi-\psi\|_r,
	\end{aligned}
	\]
	where \eqref{eq:a1ll} has been used in the last inequality and  $C$	is a positive constant. This is the desired result  \eqref{eq:mudi}. The proof is complete.	\end{proof}
\begin{remark}
	By the argument of the above theorem, we can obtain that $(W_t, \Lambda(t))$ admits the Feller property. Indeed, denote by $\{P(t,(\varphi,k),A):t\geq0,(\varphi,k)\in \mcl C_r^{2d}\times \mbb S,A\in\mathcal B(\mcl C_r^{2d}\times \mbb S)\}$ the transition probability family of the process $(W_t,\Lambda(t))$. Since $\mbb S$ has a discrete topology, we only need to prove that for each $0\leq t\leq T$, $\varphi$, $\psi\in \mcl C_r^{2d} $ and $k\in \mbb S$, $P(t,(\varphi,k),\cdot)$ converges weakly to $P(t,(\psi,k),\cdot)$ as $\|\varphi-\psi\|_r\rightarrow0$. By virtue of Theorem 5.6 in  \cite{chen2004markov}, it suffices to prove that
	\begin{equation*}
		W_1(P(t,(\varphi,k),\cdot),P(t,(\psi,k),\cdot))\rightarrow0\quad {\rm as} \quad \|\varphi-\psi\|_r \rightarrow 0,
	\end{equation*}
	where $W_1(\cdot,\cdot)$ denotes the $L^1$-Wasserstein metric between two probability measures. 
	We construct the coupling process $\left(W^{\prime}(t), \Lambda^{\prime}(t)\right)$ of $\left(W(t), \Lambda(t)\right)$. Firstly, we construct the basic coupling $(\Lambda(t),\Lambda^{\prime}(t))$ of the discrete component $\Lambda(t)$ as \cite[Example 0.20]{chen2004markov}. Then, we rewrite \eqref{eq:leilei} as 
	\begin{equation*}
		\d W(t)=D(W_t,\Lambda(t))\d t+F(W(t),\Lambda(t))\d \hat{B}(t),
	\end{equation*}
	where 
	\begin{equation*}
		D(\varphi,k)=\left(
		\begin{array}{c}
			a\varphi_1(0)+b\varphi_2(0) \\
			b_1\left(\varphi, k\right)+b_2(\varphi(0),k)
		\end{array}\right), \quad
		F(\varphi(0), k)=\left(\begin{array}{cc}
			0 & 0 \\
			0 & \sigma(\varphi(0), k)
		\end{array}\right),
	\end{equation*}
	for any $\varphi=(\varphi_1,\varphi_2)\in \mcl C_r^{2d}$ and $k\in\mbb S$, and $\hat{B}(t)$ is a \( 2 d \)-dimensional Brownian motion. For  any $(\varphi, k, \psi, l) \in \mcl C_r^{2d}\times \mathbb{S} \times \mcl C_r^{2d} \times \mathbb{S}$,  set two \( 4 d \times 4 d \) matrices as follows:
	\[
	\sigma(\varphi(0), \psi(0), k, l)=\left(\begin{array}{cc}
		F(\varphi(0), k) & 0 \\
		0 & F(\psi(0), l)
	\end{array}\right), \quad \sigma(\varphi(0), \psi(0), k)=\left(\begin{array}{cc}
		F(\varphi(0), k) & 0 \\
		F(\psi(0), k) & 0
	\end{array}\right).
	\]
	Set
	\[
	\widetilde{F}\left(t,\varphi(0), \psi(0), k, l\right)=\mathbf{1}_{[0, \tau)}(t) \sigma(\varphi(0), \psi(0), k)+\mathbf{1}_{[\tau, \infty)}(t) \sigma\left(\varphi(0), \psi(0), k, l\right)
	\]
	where $\tau=\inf \left\{t \geq 0 ; \Lambda(t)=\Lambda^{\prime}(t)\right\}$ be the coupling time of $\left(\Lambda(t), \Lambda^{\prime}(t)\right)$. Let the coupling process \( (W(t), W^{\prime}(t)) \) satisfy
	\begin{equation}\label{eq:pijingzhanji}
		\mathrm{d}\left(\begin{array}{c}
			W(t)\\W^{\prime}(t)
		\end{array}\right)=\left(\begin{array}{c}
			D\left(W_{t}, \Lambda(t)\right) \\
			D\left(W_{t}^{\prime}, \Lambda^{\prime}(t)\right)
		\end{array}\right) \mathrm{d} t+\widetilde{F}\left(t,W(t), W(t)^{\prime}, \Lambda(t), \Lambda^{\prime}(t)\right) \mathrm{d} \widetilde{B}(t),
	\end{equation}
	where \( \widetilde{B}(t) \) is a \( 4 d \)-dimensional Brownian motion. It is easy to check that $(W(t), W^{\prime}(t))$ determined by equation \eqref{eq:pijingzhanji} is the independent coupling on $[0, \tau)$ and the basic coupling on $[\tau,\infty)$.  Hence, by the definition of the $L^1$-Wasserstein metric, we only need to show that for any fixed $t>0$,
	\begin{equation}\label{eq:chngshi}
		\mathbb{E}[\lambda(W_t,\Lambda(t),W_t^{\prime},\Lambda^{\prime}(t))]\to 0\quad {\rm as} \quad  \|\varphi-\psi\|_r \rightarrow 0
	\end{equation}
	where $(W_t,W^{\prime}_t)$ is the segment process corresponding to the  coupling process determined by \eqref{eq:pijingzhanji} and $(W_0,\Lambda(0),W_0^{\prime},\Lambda^{\prime}(0))=(\varphi,k,\psi,k) $ with $\varphi\neq \psi$. Note that $\Lambda(t)$ is a Markov chain and  $\Lambda(0)=\Lambda^{\prime}(0)=k$. Hence,  for any $t\geq 0$, $\Lambda(t)=\Lambda^{\prime}(t)$ . By \eqref{eq:mudi}, we can obtain that
	\begin{equation}
		\mathbb{E}[\lambda(W_t,\Lambda(t),W_t^{\prime},\Lambda^{\prime}(t))]=\mathbb{E}\|W_t-W_t^{\prime}\|_r\leq C\|\varphi-\psi\|_r,
	\end{equation}
	which implies the desired assertion.
\end{remark}

For subsequent use, we introduce some notation. Set $D:=D([0,\infty),\mcl C_r^{2d}\times \mbb S)$ and denote by $\mathcal{D}$ the usual $\sigma$-field of $D$. Likewise, for any $T>0$, set $D_{T}:=D([0,T],\mcl C_r^{2d}\times \mbb S)$ and denote by $\mathcal{D}_{T}$ the usual $\sigma$-field of $D_{T}$. Moreover, denote by $\mu_{1}(\cdot)$ the probability distribution induced by $(Z_t,\Theta(t))$ and $\mu_{2}(\cdot)$ the probability distribution induced by $(W_t,\Lambda(t))$ in the path space $(D,\mathcal{D})$, respectively. Denote by $\mu_{1}^{T}(\cdot)$ the restriction of $\mu_{1}(\cdot)$ and $\mu_{2}^{T}(\cdot)$ the restriction of $\mu_{2}(\cdot)$ to $(D_{T},\mathcal{D}_{T})$, respectively. 

\begin{lemma}
	For any given $T>0$, $\mu_{1}^T(\cdot)$ is absolutely continuous with respect to $\mu_{2}^T(\cdot)$ and the corresponding Radon-Nikodym derivative
	$$\frac{\text{d}\mu_{1}^T}{\text{d}\mu_{2}^T}(W_{\cdot},\Lambda(\cdot))=M_{T}(W_{\cdot},\Lambda(\cdot))$$ defined in \eqref{eq:buyi}.
\end{lemma}
\begin{proof}
	According to Lemma \ref{eq:ninini} and Step 1-2 of Theorem \ref{thm:weaksolutionS}, this lemma is obtained using similar arguments as those in the proof of Lemma 4.2 in \cite{xi2009}.
\end{proof} 

We can establish the following Lemma \ref{strlem:conver} and \ref{strlem:inte} using the methods presented in Lemma 4.3 and 4.4 of \cite{xi2009}, respectively. To do this, let us give an additional assumption.
\begin{enumerate}
	\item [\textbf{(H3)}] There exists a constant $K\geq 0$ such that
	\begin{equation}
		|q_{kl}(\varphi)-q_{kl}(\psi)|\leq K\hat{q}_{kl}\|\varphi-\psi\|_r,
	\end{equation}
	for all $\varphi$, $\psi\in\mcl C_r^{2d}$ and $k\neq l\in\mbb S$.
\end{enumerate}		

\begin{lemma}\label{strlem:conver}
	Suppose that Assumptions $\textbf{(A2)-(A4)}$ and $\textbf{(H1)-(H3)}$ hold. For all $T>0$, we have that
	\begin{equation}
		\mbb E[|M_{T}(W^{(\varphi,k)}_{\cdot}, \Lambda^{(k)}(\cdot))-M_{T}(W^{(\psi,k)}_{\cdot}, \Lambda^{(k)}(\cdot))|]\rightarrow0
	\end{equation}
	as $\|\varphi-\psi\|_r\rightarrow0$.
\end{lemma}
\begin{proof}
	By \eqref{eq:buyi} and $\textbf{(A4)}$, we obtain that
	\begin{align}\label{eq:s1}
		\mbb E &\nonumber[|M_{T}(W^{(\varphi,k)}_{\cdot}, \Lambda^{(k)}(\cdot))-M_{T}(W^{(\psi,k)}_{\cdot}, \Lambda^{(k)}(\cdot))|]\\\nonumber
		&\leq\e^{HT}\mbb E\left|\prod_{i=0}^{n(T)-1} \frac{q_{\Lambda\left(\tau_{i}\right) \Lambda\left(\tau_{i+1}\right)}\left(W^{(\varphi,k)}_{\tau_{i+1}}\right)}{\hat{q}_{\Lambda\left(\tau_{i}\right) \Lambda\left(\tau_{i+1}\right)}}  \exp \left(-\int_{0}^{T}q_{\Lambda(s)}(W^{(\varphi,k)}_s)\mathrm{d}s\right)\right.\\\nonumber
		&\left.\quad-\prod_{i=0}^{n(T)-1} \frac{q_{\Lambda\left(\tau_{i}\right) \Lambda\left(\tau_{i+1}\right)}\left(W^{(\psi,k)}_{\tau_{i+1}}\right)}{\hat{q}_{\Lambda\left(\tau_{i}\right) \Lambda\left(\tau_{i+1}\right)}}  \exp \left(-\int_{0}^{T}q_{\Lambda(s)}(W^{(\psi,k)}_s)\mathrm{d}s\right)\right|\\\nonumber
		&\leq \e^{HT}\mbb E\left[\left|\prod_{i=0}^{n(T)-1} \frac{q_{\Lambda\left(\tau_{i}\right) \Lambda\left(\tau_{i+1}\right)}\left(W^{(\varphi,k)}_{\tau_{i+1}}\right)}{\hat{q}_{\Lambda\left(\tau_{i}\right) \Lambda\left(\tau_{i+1}\right)}}  \left[\exp \left(-\int_{0}^{T}q_{\Lambda(s)}(W^{(\varphi,k)}_s)\mathrm{d}s\right)-\exp \left(-\int_{0}^{T}q_{\Lambda(s)}(W^{(\psi,k)}_s)\mathrm{d}s\right)\right]\right|\right.\\\nonumber
		&\quad \left.+\left|\left(\prod_{i=0}^{n(T)-1} \frac{q_{\Lambda\left(\tau_{i}\right) \Lambda\left(\tau_{i+1}\right)}\left(W^{(\varphi,k)}_{\tau_{i+1}}\right)}{\hat{q}_{\Lambda\left(\tau_{i}\right) \Lambda\left(\tau_{i+1}\right)}}-\prod_{i=0}^{n(T)-1} \frac{q_{\Lambda\left(\tau_{i}\right) \Lambda\left(\tau_{i+1}\right)}\left(W^{(\psi,k)}_{\tau_{i+1}}\right)}{\hat{q}_{\Lambda\left(\tau_{i}\right) \Lambda\left(\tau_{i+1}\right)}}\right)  \exp \left(-\int_{0}^{T}q_{\Lambda(s)}(W^{(\psi,k)}_s)\mathrm{d}s\right)\right|\right]\\
		&\leq \e^{HT}\left[\mbb E\left|\exp \left(-\int_{0}^{T}q_{\Lambda(s)}(W^{(\varphi,k)}_s)\mathrm{d}s\right)-\exp \left(-\int_{0}^{T}q_{\Lambda(s)}(W^{(\psi,k)}_s)\mathrm{d}s\right)\right|\right. \\\nonumber
		&\left.\quad +\mbb E\left|\prod_{i=0}^{n(T)-1} \frac{q_{\Lambda\left(\tau_{i}\right) \Lambda\left(\tau_{i+1}\right)}\left(W^{(\varphi,k)}_{\tau_{i+1}}\right)}{\hat{q}_{\Lambda\left(\tau_{i}\right) \Lambda\left(\tau_{i+1}\right)}}-\prod_{i=0}^{n(T)-1} \frac{q_{\Lambda\left(\tau_{i}\right) \Lambda\left(\tau_{i+1}\right)}\left(W^{(\psi,k)}_{\tau_{i+1}}\right)}{\hat{q}_{\Lambda\left(\tau_{i}\right) \Lambda\left(\tau_{i+1}\right)}}\right|\right]\\\nonumber
		&=:\e^{HT}((\rm I)+(\rm{II})).\nonumber
	\end{align}
	For any given $\varepsilon>0$, by $\textbf{(A4),(H3)}$ and the inequality $|\exp(-a)-\exp(-b)|\leq|a-b|$ for any $a$, $b\geq0$, we obtain
	\begin{align}\label{eq:I}
		(\rm{I})&\nonumber\leq\mbb E\int_{0}^{T}\left|q_{\Lambda(s)}(W^{(\varphi,k)}_s)-q_{\Lambda(s)}(W^{(\psi,k)}_s)\right|\mathrm{d}s\nonumber\\
		&\leq\mbb E\int_{0}^{T}\sum\limits_{l\neq \Lambda(s)}\left|q_{\Lambda(s)l}(W^{(\varphi,k)}_s)-q_{\Lambda(s)l}(W^{(\psi,k)}_s)\right|\mathrm{d}s\\
		&\leq \mbb E\int_{0}^{T}\sum\limits_{l\neq \Lambda(s)}\left(\varepsilon K\hat{q}_{\Lambda(s)l}\mbf{1}_{\left\{\sup\limits_{0\leq t\leq T}\|W_t^{(\varphi,k)}-W_t^{(\psi,k)}\|_r\leq\varepsilon\right\}}+2\hat{q}_{\Lambda(s)l}\mbf{1}_{\left\{\sup\limits_{0\leq t\leq T}\|W_t^{(\varphi,k)}-W_t^{(\psi,k)}\|_r\geq\varepsilon\right\}}\right)\mathrm{d}s\nonumber\\
		&\leq  \varepsilon KHT+2HT\mbb P\left\{\sup\limits_{0\leq t\leq T}\|W_t^{(\varphi,k)}-W_t^{(\psi,k)}\|_r\geq\varepsilon\right\}.\nonumber
	\end{align}
	In addition, using the obvious inequality
	\begin{align*}
		\left|\prod_{i=1}^{m}a_i-\prod_{i=1}^{m}b_i\right|\leq m\left(\max\limits_{1\leq i\leq m}\{a_i, b_i\}\right)^{m-1}\max\limits_{1\leq i\leq m}|a_i-b_i|
	\end{align*}
	for any two positive sequences of $\{a_i\}_{i=1}^{m}$ and $\{b_i\}_{i=1}^{m}$, and by $\textbf{(H3)}$, we obtain that
	\begin{equation}\label{eq:II}
		\begin{aligned}               (\rm{II})\leq&\mbb E\left[\sum\limits_{n=0}^{\infty}n\mbf{1}_{\{n(T)=n\}}\cdot\max_{ 0\leq i\leq n-1}\left| \frac{q_{\Lambda\left(\tau_{i}\right) \Lambda\left(\tau_{i+1}\right)}\left(W^{(\varphi,k)}_{\tau_{i+1}}\right)}{\hat{q}_{\Lambda\left(\tau_{i}\right) \Lambda\left(\tau_{i+1}\right)}}-\frac{q_{\Lambda\left(\tau_{i}\right) \Lambda\left(\tau_{i+1}\right)}\left(W^{(\psi,k)}_{\tau_{i+1}}\right)}{\hat{q}_{\Lambda\left(\tau_{i}\right) \Lambda\left(\tau_{i+1}\right)}}\right|\right]\\
			\leq &\mbb E\left[\sum\limits_{n=0}^{\infty}n\mbf{1}_{\{n(T)=n\}}\left(\varepsilon K\mbf{1}_{\left\{\sup\limits_{0\leq t\leq T}\|W_t^{(\varphi,k)}-W_t^{(\psi,k)}\|_r\leq\varepsilon\right\}}+2\mbf{1}_{\left\{\sup\limits_{0\leq t\leq T}\|W_t^{(\varphi,k)}-W_t^{(\psi,k)}\|_r\geq\varepsilon\right\}}\right)\right]\\
			\leq &\varepsilon K\mbb E[n(T)]+2\mbb E\left[n(T)\mbf{1}_{\left\{\sup\limits_{0\leq t\leq T}\|W_t^{(\varphi,k)}-W_t^{(\psi,k)}\|_r\geq\varepsilon\right\}}\right].
		\end{aligned}
	\end{equation}
	Note that $n(T)$ is the number of jumps of trajectories of $\Lambda(t)$ prior to $T$. Hence we have
	\begin{equation}\label{eq:shuijiao}
		\begin{aligned}
			\mbb E[n(T)]=&\mbb E\left[\sum\limits_{s\leq T}\mbf{1}_{\{\Lambda(s)\neq\Lambda(s-)\}}\right]
			=\sum\limits_{s\leq T}\mbb E\left[\mbb E\left[\mbf{1}_{\{\Lambda(s)\neq\Lambda(s-)\}}|\Lambda(s-)\right]\right]\\
			=&\sum\limits_{s\leq T}\mbb E\left[\sum\limits_{l\neq\Lambda(s-)}\mbb P\left(\Lambda(s)=l|\Lambda(s-)\right)\right]=\sum\limits_{s\leq T}\mbb E\left[\sum\limits_{l\neq\Lambda(s-)}\hat{q}_{\Lambda(s-)l}\right]\\
			=&\mbb E\left[\int_{0}^{T}\hat{q}_{\Lambda(s-)}\d s\right]\leq HT<\infty.
		\end{aligned}  
	\end{equation}
	
	Combining \eqref{eq:s1}-\eqref{eq:shuijiao}  leads to 
	\begin{equation}\label{eq:zuihou}
		\begin{aligned}
			\mbb E &[|M_{T}(W^{(\varphi,k)}_{\cdot}, \Lambda^{(k)}(\cdot))-M_{T}(W^{(\psi,k)}_{\cdot}, \Lambda^{(k)}(\cdot))|]\\
			&\leq  2\varepsilon KHT+2HT\mbb P\left\{\sup\limits_{0\leq t\leq T}\|W_t^{(\varphi,k)}-W_t^{(\psi,k)}\|_r\geq\varepsilon\right\} + 2\mbb E\left[n(T)\mbf{1}_{\left\{\sup\limits_{0\leq t\leq T}\|W_t^{(\varphi,k)}-W_t^{(\psi,k)}\|_r\geq\varepsilon\right\}}\right].
		\end{aligned}
	\end{equation}
	From this and Lemma \ref{strlem:basic}, we derive that the second and third term in \eqref{eq:zuihou} tends to zero as $\|\varphi-\psi\|_r\rightarrow0$. At the same time, we also know that the first term in \eqref{eq:zuihou} can be arbitrarily small since the multiplier $\varepsilon$ is arbitrary. The proof is complete.
\end{proof}

\begin{lemma}\label{strlem:inte}
	Suppose that Assumption $\textbf{(A4)}$ holds. For all $T>0$ and $(\varphi,k)\in \mcl C_r^{2d}\times \mbb S$, $M_{T}(W^{(\varphi,k)}_{\cdot}, \Lambda^{(k)}(\cdot))$ is integrable.
\end{lemma}
\begin{proof}
	By $\textbf{(A4)}$ and \eqref{eq:buyi}, we obtain that
	\begin{align*}
		\mbb E [M_{T}(W^{(\varphi,k)}_{\cdot}, \Lambda^{(k)}(\cdot))]=&\mbb E\left[\prod_{i=0}^{n(T)-1} \frac{q_{\Lambda\left(\tau_{i}\right) \Lambda\left(\tau_{i+1}\right)}\left(W^{(\varphi,k)}_{\tau_{i+1}}\right)}{\hat{q}_{\Lambda\left(\tau_{i}\right) \Lambda\left(\tau_{i+1}\right)}}  \exp \left(-\sum_{i=0}^{n(T)} \int_{\tau_{i}}^{\tau_{i+1} \wedge T}\left[q_{\Lambda\left(\tau_{i}\right)}(W^{(\varphi,k)}_s)-\hat{q}_{\Lambda\left(\tau_{i}\right)}\right] \mathrm{d} s\right)\right]\\
		\leq&\e^{HT}\mbb E\left[\prod_{i=0}^{n(T)-1} \frac{q_{\Lambda\left(\tau_{i}\right) \Lambda\left(\tau_{i+1}\right)}\left(W^{(\varphi,k)}_{\tau_{i+1}}\right)}{\hat{q}_{\Lambda\left(\tau_{i}\right) \Lambda\left(\tau_{i+1}\right)}}\right]\leq \e^{HT}< +\infty,
	\end{align*}
	which implies the desired integrability. This completes the proof.
\end{proof}
Now we present the main results of this section.
\begin{theorem}\label{thm:strongfeller}
	Suppose that Assumptions $\textbf{(A2)-(A4)}$ and $\textbf{(H1)-(H3)}$ hold. Then $(Z_t, \Theta(t))$ has the Feller property.
\end{theorem}
\begin{proof}
	To prove the desired Feller property, it is enough to prove that for any $t>0$ and any bounded continuous function $f$ on $\mcl C_r^{2d}\times \mbb S$, $ \mbb E[f(Z_t^{(\varphi,k)},\Theta^{(k)}(t))]$ is bounded continuous in both $\varphi\in \mcl C_r^{2d}$ and $k\in \mbb S$. Since $\mbb S$ has a discrete metric, it is sufficient to prove that
	\begin{equation*}
		\left|\mbb E\left[f\left(Z_t^{(\varphi,k)},\Theta^{(k)}(t)\right)\right]-\mbb E\left[f\left(Z_t^{(\psi,k)},\Theta^{(k)}(t)\right)\right]\right|\rightarrow0
	\end{equation*}
	as $\|\varphi-\psi\|_r\rightarrow0$. Indeed, by the Radon-Nikodym derivative, for all $(\varphi,k)\in \mcl C_r^{2d}\times \mbb S$,
	\begin{equation*}
		\begin{split}
			&\mbb E\left[f\left(Z_t^{(\varphi,k)},\Theta^{(k)}(t)\right)\right]=\mbb E\left[f\left(W_t^{(\varphi,k)},\Lambda^{(k)}(t)\right)\cdot M_{t}\left(W_{\cdot}^{(\varphi,k)},\Lambda^{(k)}(\cdot)\right)\right].
		\end{split}
	\end{equation*}
	It follows from Lemma \ref{strlem:basic} that
	\begin{equation}\label{eq:12}
		f(W^{(\varphi,k)}_t,\Lambda^{(k)}(t))\rightarrow f(W^{(\psi,k)}_t,\Lambda^{(k)}(t))\qquad \text{in probability}
	\end{equation}
	as $\|\varphi-\psi\|_r\rightarrow0$. Therefore, for any given $\varepsilon>0$, we have
	\begin{equation}\label{streq:longeq}
		\begin{aligned}
			\Big|\mbb E&\left[f\left(Z_t^{(\varphi,k)},\Theta^{(k)}(t)\right)\right]-\mbb E\left[f\left(Z_t^{(\psi,k)},\Theta^{(k)}(t)\right)\right]\Big|\\
			&\leq \mbb E\left|f\left(W_t^{(\varphi,k)},\Lambda^{(k)}(t)\right)\cdot M_{t}\left(W_{\cdot}^{(\varphi,k)},\Lambda^{(k)}(\cdot)\right)-f\left(W_t^{(\psi,k)},\Lambda^{(k)}(t)\right)\cdot M_{t}\left(W_{\cdot}^{(\psi,k)},\Lambda^{(k)}(\cdot)\right)\right|\\
			&\leq\|f\| \mbb E\left|M_{t}\left(W_{\cdot}^{(\varphi,k)},\Lambda^{(k)}(\cdot)\right)-M_{t}\left(W_{\cdot}^{(\psi,k)},\Lambda^{(k)}(\cdot)\right)\right|\\
			&\quad+2\|f\| \mbb E\left[M_{t}\left(W_{\cdot}^{(\psi,k)},\Lambda^{(k)}(\cdot)\right)\mathbf{1}_{\{|f\left(W_t^{(\varphi,k)},\Lambda^{(k)}(t)\right)-f\left(W_t^{(\psi,k)},\Lambda^{(k)}(t)\right)|\geq\varepsilon\}}\right]\\
			&\quad+\varepsilon \mbb E\left[M_{t}\left(W_{\cdot}^{(\psi,k)},\Lambda^{(k)}(\cdot)\right)\right]\\
			&=\mathrm{(I)}+\mathrm{(II)}+\mathrm{(III)},
		\end{aligned}
	\end{equation}
	where $\|f\|:=\sup\{|f(\varphi,k)|:(\varphi,k)\in \mcl C_r^{2d}\times \mbb S\}$. From Lemma \ref{strlem:conver}, term $\mathrm{(I)}$ in \eqref{streq:longeq} tends to zero as $\|\varphi-\psi\|_r\rightarrow0$. From Lemma \ref{strlem:inte} and \eqref{eq:12}, we derive that term $\mathrm{(II)}$ in \eqref{streq:longeq}  also tends to zero as $\|\varphi-\psi\|_r\rightarrow0$. Meanwhile, term $\mathrm{(III)}$ in \eqref{streq:longeq}  can be arbitrarily small since the multiplier $\varepsilon$ is arbitrary and $M_{t}(W_{\cdot}^{(\psi,k)},\Lambda^{(k)}(\cdot))$ is integrable by Lemma \ref{strlem:inte}. The proof is complete.
\end{proof}

\section{Concluding Remarks}
Motivated by the increasing need of modeling complex systems, this paper is devoted to the investigation of a class of functional stochastic Hamiltonian systems with singular coefficients and countable regimes. By using Girsanov’s transformation, we have established the martingale solution of the system in each fixed switching case. Since the discrete component is dependent of the continuous component, we have investigated the special case when the discrete component is independent of the continuous component. Furthermore, with the help of a martingale process, we have obtained the well-posedness of the system. Moreover, we have used appropriate Radon-Nikodym derivative and a vector valued elliptic equation to derive the Feller property.

A number of other problems deserve further investigation. For example, when the history appears in the diffusion coefficient $\sigma$, the system becomes more difficult to deal with. Additionally, the work on the asymptotically strong Feller property and exponential ergodicity for the segment process is still scare but is required in numerous applications. This will be considered in our future work.

\section*{Declarations}
\noindent Funding: This work was supported in part by the National Natural Science Foundation of China under Grant No. 12071031.

\bigskip\noindent
Conflict of interest/Competing interests: All authors certify that they have no affiliations with or involvement in any organization or entity with any financial interest or non-financial interest in the subject matter or materials discussed in this manuscript.

\bigskip\noindent
Data availability: Data sharing is not applicable to this article as no datasets were generated or analyzed during the current study.

\bibliographystyle{plain}
\bibliography{LaTeXSource}

\end{document}